\documentclass[11pt]{article}

\usepackage{amsmath}
\usepackage{amsfonts}
\usepackage{graphicx}
\usepackage{setspace}
\usepackage{amsmath}
\usepackage{amssymb}
\usepackage{latexsym}
\usepackage{amsmath, amsfonts,amssymb, amsthm, euscript,makeidx,color,mathrsfs}

\usepackage[numbers,sort&compress]{natbib}

\def\sqr#1#2{{\vcenter{\vbox{\hrule height.#2pt
              \hbox{\vrule width.#2pt height#1pt \kern#1pt \vrule width.#2pt}
              \hrule height.#2pt}}}}
\usepackage{hyperref}
\hypersetup{
    colorlinks=true,
    linkcolor=blue,
    urlcolor=cyan,
    citecolor=cyan
}
\usepackage{caption}

\RequirePackage[capitalize,nameinlink]{cleveref}

\crefname{section}{section}{sections}
\crefname{subsection}{subsection}{subsections}
\Crefname{section}{Section}{Sections}
\Crefname{subsection}{Subsection}{Subsections}

\crefname{condition}{Condition}{Conditions}

\Crefname{figure}{Figure}{Figures}

\crefformat{equation}{\textup{#2(#1)#3}}
\crefrangeformat{equation}{\textup{#3(#1)#4--#5(#2)#6}}
\crefmultiformat{equation}{\textup{#2(#1)#3}}{ and \textup{#2(#1)#3}}
{, \textup{#2(#1)#3}}{ and \textup{#2(#1)#3}}
\crefrangemultiformat{equation}{\textup{#3(#1)#4--#5(#2)#6}}%
{ and \textup{#3(#1)#4--#5(#2)#6}}{, \textup{#3(#1)#4--#5(#2)#6}}{ and \textup{#3(#1)#4--#5(#2)#6}}

\Crefformat{equation}{#2Equation~\textup{(#1)}#3}
\Crefrangeformat{equation}{Equations~\textup{#3(#1)#4--#5(#2)#6}}
\Crefmultiformat{equation}{Equations~\textup{#2(#1)#3}}{ and \textup{#2(#1)#3}}
{, \textup{#2(#1)#3}}{ and \textup{#2(#1)#3}}
\Crefrangemultiformat{equation}{Equations~\textup{#3(#1)#4--#5(#2)#6}}%
{ and \textup{#3(#1)#4--#5(#2)#6}}{, \textup{#3(#1)#4--#5(#2)#6}}{ and \textup{#3(#1)#4--#5(#2)#6}}

\crefdefaultlabelformat{#2\textup{#1}#3}

\newtheorem {theorem}{Theorem}[section]
\newtheorem {lemma}[theorem]{{\bf Lemma}}

\newtheorem {proposition}[theorem]{{\bf Proposition}}
\theoremstyle{remark}
\newtheorem {remark}{{\bf Remark}}[section]

\theoremstyle{definition}
\newtheorem {definition}{{\bf Definition}}[section]

\theoremstyle{plain} \numberwithin {equation}{section}
\newtheorem{assumption}{Assumption}

\numberwithin{assumption}{section}

\allowdisplaybreaks

\usepackage{graphicx}
\usepackage{ifpdf}
\ifpdf
  \usepackage{epstopdf}
\fi

\usepackage{enumerate}

\def\deq{\mathop{\buildrel\Delta\over=}}

\begin{document}

\title{Null controllability for stochastic coupled systems of   fourth order parabolic equations}
\author{Yu Wang\footnote{School of Mathematics, Southwest Jiaotong University, Chengdu, P. R. China.
Email: yuwangmath@163.com.}
}
\date{}

\maketitle

\begin{abstract}
	This paper aims to establish null controllability for systems coupled by two backward fourth order stochastic parabolic equations.
    The main goal is to control both equations with only one control act on the drift term.
    To achieve this, we develop a new global Carleman estimate for fourth order stochastic parabolic equations, which allows us to deduce a suitable observability inequality for the adjoint systems. 
\end{abstract}



\section{Introduction}

Let $T>0$ and $(\Omega, \mathcal{F}, \mathbf{F},   \mathbb{P})$ with $\mathbf   F=\{\mathcal{F}_{t}\}_{t \geq 0}$ be a complete filtered probability space on which a   one-dimensional standard Brownian motion   $\{W(t)\}_{t \geq 0}$ is defined and   $\mathbf{F}$ is the natural filtration generated   by $W(\cdot)$, augmented by all the $\mathbb{P}$   null sets in $\mathcal{F}$.
Denote by $ \mathbb{F} $    the progressive $\sigma$-field with respect   to $\mathbf{F}$.
Let $H$ be a Banach space.
Write  $L^2_{\mathcal{F}_t}(\Omega;H)$ for the  space of all $\mathcal{F}_t$-measurable random   variables $\xi$ such that   $\mathbb{E}|\xi|_H^2<\infty$;    $L_{\mathbb{F}}^{2}(0, T ; H)$ for the space   consisting of all $H$-valued $\mathbf F$-adapted   processes $X(\cdot)$ such that   $\mathbb{E}\bigl(|X(\cdot)|_{L^{2}(0, T ;   H)}^{2}\bigr)<\infty$;    $L_{\mathbb{F}}^{\infty}(0, T ; H)$ for  the space   consisting of all $H$-valued $\mathbf F$-adapted   bounded processes.
All these spaces are  Banach spaces with the canonical norms (e.g.,   \cite[Section 2.6]{Lue2021a}).

Let $ G \subset \mathbb{R}^{n} (n \in \mathbb{N}) $ be a bounded open set with a $ C^{4} $ boundary $ \Gamma $. 
Let $ G_{0} \subset G $ be a nonempty subset. 
Set $Q=(0, T) \times G$ and $   \Sigma= (0, T) \times \Gamma$. Denote by $\nu(x) = (\nu^1(x), \cdots, \nu^n(x))$ the unit outward normal vector of $\Gamma$ at point $x$. 
As usual, $ \chi_{G_{0}} $   denotes the characteristic function of $ G_{0} $.

Consider the following coupled fourth order backward stochastic parabolic equations:
\begin{align}\label{eqEquations}
    \left\{\begin{aligned} 
        & d y_{1} -\Delta^{2} y_{1} d t  =  ( a_{1} y_{1} + a_{2} y_{2} + a_{3} Y_{1}) d t + Y_{1} d W(t)   &&  \text {in } Q, \\
        & d y_{2} -\Delta^{2} y_{2} d t  =  ( b_{1} y_{1} + b_{2} y_{2} + b_{3} Y_{2} + \chi_{G_{0}} u) d t + Y_{2} d W(t)   &&  \text {in } Q, \\
        & y_{1}= \frac{\partial y_{1} }{\partial \nu}  =0, \quad y_{2}= \frac{\partial y_{2} }{\partial \nu}  =0  &&\text {on } \Sigma, \\
        & y_{1}(T)  = y_{1}^{T}, \quad  y_{2}(T)  = y_{2}^{T} &&  \text {in } G,
        \end{aligned}\right.
\end{align}
where $ (y_{1}^{T}, y_{2}^{T}) \in \big[ L^{2}_{\mathcal{F}_{T}}(\Omega; L^{2}(G)  ) \big]^{2} $ is a terminal state, $ u   \in L_{\mathbb{F}}^{2} (0, T ; L^{2}(G_{0}) ) $ is a control and
\begin{align*}
    \left\{ 
       \begin{aligned}
        & a_{1}, a_{2} \in L_{\mathbb{F}}^{\infty} (0, T ; L^{\infty}(G) ),  \quad   \quad  a_3 \in L_{\mathbb{F}}^{\infty} (0, T ; W^{2, \infty}(G) ), \\
        & b_{1}, b_{2} \in L_{\mathbb{F}}^{\infty} (0, T ; L^{\infty}(G) ), \quad   \quad b_3 \in L_{\mathbb{F}}^{\infty} (0, T ; W^{2, \infty}(G) )
        .
       \end{aligned}
     \right.
\end{align*}
By the classical wellposedness result for backward stochastic evolution equations (e.g., \cite[Theorem 4.11]{Lue2021a}), we know that  \cref{eqEquations} admits a unique weak solution
\begin{align*}
    (y_{1}, y_{2}; Y_{1}, Y_{2})\in \big[ L^2_{\mathbb{F}}(\Omega;C([0,T];L^2(G))) \cap L^2_{\mathbb{F}}(0,T;H^2(G)\cap H_0^1(G)) \big]^{2} \times \big[L^2_{\mathbb{F}}(0,T;L^2(G))\big]^{2}
    .
\end{align*}

Now we give the definition of the null controllability for \cref{eqEquations}.
\begin{definition}
    \label{def.control}
    The system \cref{eqEquations} is called null controllable at a  time $ T > 0$ if for any
    given $ (y_{1}^{T}, y_{2}^{T}) \in \big[ L^{2}_{\mathcal{F}_{T}}(\Omega; L^{2}(G)  ) \big]^{2}$, 
    there exists a control $ u   \in L_{\mathbb{F}}^{2} (0, T ; L^{2}(G_{0}) ) $
    such that the solution $ (y_{1}, y_{2}; Y_{1}, Y_{2})$ to \cref{eqEquations} satisfies that $ (y_{1}(0 ), y_{2}(0)) = (0,0) $, $ \mathbb{P} $-a.s.
\end{definition}

Note that we only act one control in the second equation. To control the entire system, we expect that the action of $ y_{2} $ to $ y_{1} $ is sufficiently effective.
Hence, it is reasonable to make the following assumption.
\begin{assumption} \label{as1}
    There exists a nonempty open set $ G_{1} \subset G_{0} $ and a constant $ \sigma > 0 $ such that $ a_{2}(x, t) \geq \sigma $ or $ a_{2}(x, t) \leq -  \sigma $, a.e. $ (x, t) \in G_{1} \times (0,T) $, $ \mathbb{P} $-a.s.
\end{assumption}

The main result of this paper is the following.

\begin{theorem} \label{thmControlResult}
    Under the \cref{as1}, system \cref{eqEquations} is null controllable at any time $ T > 0 $.
\end{theorem}

\begin{remark}
    From the proof of \cref{thmObservabilityEstimateI},  it is not difficult to see that  \cref{thmControlResult}  still holds when two equations of \cref{eqEquations} contains $ \nabla y_{1} $,  $ \nabla^{2} y_{1} $, $ \nabla \Delta y_{1} $ and $ Y_{1} $ in the drift terms. 
    Indeed, for instance, by adding $ Y_{1} $ to the drift term of the second equation in \cref{eqEquations}, direct computation reveals that \cref{prop1} and \cref{propTwotoOne} still hold. Therefore, \cref{thmControlResult} remains valid.
    We do not introduce these  terms in the  system  \cref{eqEquations} for the simplicity of notations. 
\end{remark}

\begin{remark}
    Using the Carleman inequality \cref{eq.carlemanEstimate} and a proof similar to \cref{thmObservabilityEstimateI}, we can establish the controllability of the backward fourth order stochastic parabolic equation.
    That is, for any $ y_{T} \in  L^{2}_{\mathcal{F}_{T}}(\Omega; L^{2}(G)  ) $, there exists a control $ u  \in L_{\mathbb{F}}^{2} (0, T ; L^{2}(G_{0}) )  $ such that the solution $ y $ of equation
    \begin{equation*} 
        \left\{\begin{alignedat}{2}
            &d y-\Delta^2 y d t=(b_{2} y+b_{3} Y+\chi_{G_{0}} u) d t+Y d W(t) && \quad \text { in } Q, \\
            &y=  \frac{\partial y}{\partial \nu}=0 &&\quad \text { on } \Sigma, \\
            &y(T)=y_{T} &&\quad \text { in } G.
        \end{alignedat}\right.
    \end{equation*}
    satisfying $ y(0) =0 $, $ \mathbb{P} $-a.s.
\end{remark}

\begin{remark}
    It would be quite interesting to study whether the system \cref{eqEquations} remains controllable if they are coupled through diffusion terms, i.e., by adding $ Y_{2} $ in the drift term of the first equation and assuming $ a_{2} \equiv 0 $. 
    We currently do not know how to prove it.
\end{remark}

To prove the null controllability of \cref{eqEquations}, we introduce the adjoint equation of \cref{eqEquations}:
\begin{align}\label{eqEquationsAdj}
    \left\{\begin{aligned} 
        & d z_{1} +\Delta^{2} z_{1} d t  =  - ( a_{1} z_{1} + b_{1} z_{2} ) d t  - a_{3} z_{1} d W(t)   &&  \text {in } Q, \\
        & d z_{2} + \Delta^{2} z_{2} d t  =  -  ( a_{2} z_{1} + b_{2} z_{2}  ) d t - b_{3} z_{2} d W(t)   &&  \text {in } Q, \\
        & z_{1}= \frac{\partial z_{1} }{\partial \nu} =0, \quad z_{2}= \frac{\partial z_{2} }{\partial \nu}  =0  &&\text {on } \Sigma, \\
        & z_{1}(0)  = z_{1}^{0}, \quad  z_{2}(0)  = z_{2}^{0} &&  \text {in } G.
        \end{aligned}\right.
\end{align}
By the classical wellposedness result for stochastic evolution equations (e.g., \cite[Theorem 3.24]{Lue2021a}), we know that \cref{eqEquationsAdj}  admits a unique weak solution
\begin{align*}
    (z_{1}, z_{2}) \in \big[L^2_{\mathbb{F}}(\Omega;C([0,T];L^2(G)))\cap L^2_{\mathbb{F}}(0,T;H^2(G)\cap H_0^1(G)) \big]^{2}.
\end{align*}

By a standard duality argument, \Cref{thmControlResult} is implied by the following observability estimate (e.g., \cite[Theorem 7.17]{Lue2021a}).
\begin{theorem} \label{thmObservabilityEstimateI}
    There exists a constant $ C > 0 $ such that the solution $ (z_{1}, z_{2}) $ to system \cref{eqEquationsAdj} satisfies
    \begin{align*}
        |(z_{1}(T), z_{2}(T))|_{L^{2}_{\mathcal{F}_{T}}(\Omega; L^{2}(G) \times L^{2}(G))} 
        \leq 
        C |z_{2}|_{L^{2}_{\mathbb{F}}(0,T; L^{2}(G_{0}))}
        .
    \end{align*}
\end{theorem}
Here and in what follows,
we denote by $C$  a generic positive constant depending on $G$, $G_{0}$, $T$ and $ a_{i}, b_{i} $, $ i=1, 2, 3 $, whose value may vary from line to line.

A plenty  of works on controllability and observability estimates for coupled deterministic parabolic systems  exist in the literature, including works such as \cite{Guerrero2007,Lissy2019,Duprez2016,Carreno2016,Steeves2019,Steeves2019a} and the references therein.
The main issue, compared to the single-equation cases, is that one wants to control the system with fewer controls than equations by indirectly affecting the equations where no control term appears, thanks to the coupling terms in the system.
Hence, it is more difficult to obtain the controllability for coupled systems than for the single equations.
It is a fundamental problem since, in some complex systems, only a few quantities can be effectively controlled.
Additional details regarding this specific subject can be found in the survey article \cite{AmmarKhodja2011}.

As far as we know, there are only a few works studying controllability problems for coupled stochastic systems with fewer controls than equations, see, \cite{li2012,HernandezSantamaria2022,Liu2018,Liu2014a}. 
In particular, in \cite{HernandezSantamaria2022}, the authors obtain null controllability for one dimensional linear backward stochastic systems coupling fourth- and second-order parabolic equations. 

In order to prove \cref{thmObservabilityEstimateI}, we will use Carleman estimates, which is derived taking inspiration from the literature \cite{Gao2015,Guerrero2019,Liu2014,Tang2009}.
To state the Carleman estimate, we need to introduce the weight functions:
\begin{equation}\label{eq.alpha}
    \alpha(x, t)=\frac{e^{\lambda(|2
    \eta|_{C(\overline{G})}+\eta(x))}-e^{4
    \lambda|\eta|_{C(\overline{G})}}}{t^{1/2}(T-t)^{1/2}},
    \quad \xi(x, t)=\frac{e^{\lambda(|2
    \eta|_{C(\overline{G})}+\eta(x))}}{t^{1/2}(T-t)^{1/2}}.
\end{equation}
Let $ G_{i} $  be three nonempty open sets such that $ \overline{G_{i}} \subset G_{i-1} $  and denote by $ Q_{i} = (0,T) \times G_{i} $ for $ i = 1,2,3 $.
Let the function $ \eta (x) $ satisfy
\begin{align}\label{eq.eta} 
    \eta \in C^{4}(\overline{G}),\quad  \eta >0
\text{ in } G, \quad \eta|_{{\Gamma}}=0,
\quad|\nabla \eta| \geq C>0 \text { in }
\overline{ G \backslash G_{3}} 
\end{align}
For the existence of $ \eta $, see \cite[Lemma 1.1]{Fursikov1996}.
We remark that our weight functions are the same as that in \cite{Wang}.

We have the following Carleman estimate.
\begin{theorem}\label{thm.carlemanForwardPre}
	For any $ T > 0 $, there are positive constants $ C  , \lambda_{0}  $ and $ s_{0} = s_{0} (T)   $,  such that for all $ \lambda \geq \lambda_{0} $,  $s \geq s_{0}$  and
	$ \varphi  \in L^{2}_{\mathbb{F}}(0,T;H^{2}(G)),    f\in L^{2}_{\mathbb{F}}(0,T;L^{2}(G)),  g \in L^{2}_{\mathbb{F}}(0,T;H^{2}(G)) $
	satisfying
	\begin{equation}\label{eq.carlemanForwardEquation}
		\left\{
		\begin{alignedat}{2}
			&d \varphi +\Delta^2 \varphi  d t=f d t + g   d W(t) && \quad \text { in } Q, \\
			&\varphi = \frac{\partial \varphi }{\partial \nu} =0 && \quad \text { on } \Sigma,
		\end{alignedat}
		\right.
	\end{equation}
	it holds that 
	\begin{align}\label{eq.carlemanEstimate}
		\notag
		& 
        \mathbb{E} \int_{Q} e^{2 s \alpha}(
            s^{6} \lambda^{8} \xi^{6} | \varphi |^{2}
        +s^{4} \lambda^{6} \xi^{4} |\nabla \varphi |^{2} 
        +s^{3} \lambda^{4}    \xi^{3} |\Delta \varphi |^{2} 
        + s^{2}  \lambda^{4}  \xi^{2} | \nabla^{2} \varphi  |^{2}   
        +   \lambda    | \nabla \Delta \varphi  |^{2}   
        ) d x d t
        \\
        &
        \leq C   \Big( \mathbb{E} \int_{Q_{2}}  s^{7}  \lambda^{8} \xi^{7}  e^{2 s \alpha} |\varphi|^{2} d x d t
        +   \mathbb{E} \int_{Q} e^{2 s \alpha}  
        ( 
            |f|^{2} 
            + s^{5} \lambda^{5} \xi^{5} |g| ^{2}
            + s^{3} \lambda^{3} \xi^{3}|\nabla g|  ^{2}
            + s \lambda \xi |\nabla^{2} g| ^{2}
        )d x d t
        \Big)
        .
	\end{align}
\end{theorem}

\begin{remark}
    The orders of some terms in \cref{eq.carlemanEstimate} seem inconsistent.
    For instance, it is an interesting question whether $ \lambda^2|\nabla \Delta \varphi|^2 $ can be included on the left-hand side of \cref{eq.carlemanEstimate}.
    To do this, one possible approach is to replace $ \Phi_{2} $ in \cref{thm.fundamentalIndentity} with $ \Phi_{2} = -2s\lambda^{2}\xi |\nabla \eta|^{2} - 2s\lambda\xi \Delta \eta - \lambda^{2} $.
    This will result in the term $-2s^{2}\lambda^{4}\xi^{2}|\nabla\eta|^{2}|\Delta\psi|^{2}$ appearing in $M$ (see \cref{eq.paM1}), and we currently lack a method to handle this term.
    Hence, we currently do not know how to achieve $ \lambda^2|\nabla \Delta \varphi|^2 $.

    Moreover, the term $ s^{7}  \lambda^{8} \xi^{7}  e^{2 s \alpha} |\varphi|^{2} $ appears on the right-hand side  instead of $ s^{6}  \lambda^{8} \xi^{6}  e^{2 s \alpha} |\varphi|^{2} $. 
    It is because we need to estimate local term $   \mathbb{E} \int_{(0, T) \times G_{3}} s^{3} \lambda^{4} \xi^{3}\theta^{2}|\Delta \varphi|^{2}d x d t $ (see \cref{eq.pfthm2.1.S34}).
    Note that the term $ s^{6}  \lambda^{8} \xi^{6}  e^{2 s \alpha} |\varphi|^{2} $ on the left-hand side has a lower order of $ s $ than that on the right-hand side.
    This is because we utilize \cref{lemma.LaplaceCarleman} to obtain this term (see \cref{eq.thm31S5E8_1}), hence, we cannot obtain results of a higher order in $s$.
    Indeed, as mentioned in \cite[Remark 1.7]{LeRousseau2020}, for a fourth-order equation, it seems difficult to let $  s^{7}  \lambda^{8} \xi^{7}  e^{2 s \alpha} |\varphi|^{2} $ appear in the left hand side of \cref{eq.carlemanEstimate}.
\end{remark}

\begin{remark}
    From the proof of \Cref{thm.carlemanForwardPre} and that of \cite[Theorem 1.3]{Wang}, it is not difficult to see that Carleman inequality \cref{eq.carlemanEstimate} is still valid when $ \varphi $ satisfies the boundary condition $ \varphi= \Delta \varphi =0  $ on $ \Sigma $. 
    Hence, \cref{thmControlResult} holds true when $ y_{i} $ satisfies $ y_{i}= \Delta y_{i} =0  $ on $ \Sigma $ for $ i= 1,2 $.
\end{remark}

\begin{remark}
    Using the pointwise weighted identity  \cref{eq.finalEquation} with $ \kappa = - 1 $, we can deduce the Carleman estimate for the backward fourth order stochastic parabolic equations. 
    In comparison to existing results (see \cite[Theorem 1.3]{Wang}), we further estimate $ \nabla \Delta \varphi $, enabling us to contain $ \nabla \Delta y $ in the drift term of the control system in \cite{Wang} .
\end{remark}

The rest of this paper is organized as follows. In  \Cref{sc.2}, we give some preliminaries. \Cref{sc.3} is devoted to prove \cref{thm.carlemanForwardPre}. 
By means of this Carleman estimate, we prove \cref{thmObservabilityEstimateI} in  \Cref{sc.4}. 
Then, \Cref{sc.5} is addressed to the proof of  \cref{thmControlResult}.
Finally, we finished our article with the appendix in which we prove lemmas that are useful in deriving the Carleman estimate.

\section{Preliminary results} \label{sc.2}

In this section, we present some preliminary results for the proof of  \cref{thm.carlemanForwardPre}.

We have a pointwise weighted identity for the stochastic   partial differential operator  $ \kappa d \varphi+\Delta^{2} \varphi dt $,  which plays a fundamental role in the proof of the Carleman estimates.  

In what follows, we will use the notation $y_{x_{i}} \equiv y_{x_{i}}(x)=\partial y(x) / \partial x_{i}$, where $x_{i}$ is the $i$-th coordinate of a generic point $x=(x_{1}, \ldots, x_{n})$ in $\mathbb{R}^{n} $.

\begin{theorem}\label{thm.fundamentalIndentity}
    Let $ \varphi $ be an  $ H^{4}(G) $-valued It\^o process.
    Set   $ \theta=e^{\ell} $, $ \ell=s \alpha $, $ \psi=\theta \varphi $. 
    $ \kappa $ is a nonzero constant.
    The function $ \alpha $ and $ \xi $ are as in  \cref{eq.alpha}. 
    Let $\eta\in C^4(\mathbb{R}^n)$. 
    Then, for any $ t\in[0,T] $ and a.e.  $ (x,\omega) \in G \times \Omega $, 
    \begin{align}\label{eq.finalEquation}
        \notag
		& 2 \theta K_{1}  ( \kappa d \varphi \!+\! \Delta^{2} \varphi d t )
		 \!-\! 2 \kappa \sum_{i,j=1}^{n} (
		\psi_{x_{i}x_{i}x_{j}} d \psi - \psi_{x_{i}x_{j}}d \psi_{x_{i}}
		\!+\!\Psi_{2} \psi_{x_{i}} \delta_{ij} d \psi
		\!+\!\Psi_{3}^{ij} \psi_{x_{i}} d \psi
		)_{x_{j}}
		\!-\! 2 \operatorname{div}  (V_{1} +  V_{2}) d t
		\\ \notag
		& =
		2 K_{1} ^{2} d t
		+ 2 K_{1} K_{3}
		+ 2 M d t
		+ 2 \sum_{i,j,k,l=1}^{n} \Lambda^{i j k l}_{1} \psi_{x_{i}x_{j}}\psi_{x_{k}x_{l}}d t
		+ 2 \sum_{i,j=1}^{n} \Lambda_{2}^{i j} \psi_{x_{i}} \psi_{x_{j}}d t
        + 2 \Lambda_{3} \psi^{2}d t
		\\ \notag
		& \quad 
        \!+\!  2 \kappa d \bigg(
            \frac{1}{2} |\nabla^{2} \psi|^{2}
            \!-\!\frac{1}{2} \Psi_{2} |\nabla \psi |^{2}
            \!-\! 2 s^{2} \lambda^{2} \xi^{2} |\nabla \eta \nabla \psi|^{2}
            \!+\! \frac{1}{2} \Psi_{5} \psi^{2}
		\bigg)
		\!+\!  \kappa \theta^{2}  \Psi_{2} | d \nabla \varphi + \nabla \ell d \varphi|^{2}
		\!-\! \kappa  \Psi_{5} \theta^{2} ( d \varphi )^{2}
		\\ 
		& \quad
		+ 4\kappa  s^{2}\lambda^{2}\xi^{2} \theta^{2} | \nabla \eta d \nabla \varphi + \nabla \eta \nabla \ell d \varphi|^{2}
		-\kappa \theta^{2} \sum_{i,j=1}^{n} [  d \varphi_{x_{i}x_{j}} + 2 \ell_{x_{j}} d \varphi_{x_{i}} +   ( \ell_{x_{i}x_{j}} + \ell_{x_{i}} \ell_{x_{j}} ) d \varphi ]^{2}
    \end{align}
    Here
\begin{equation*}
    K_{1} = \Delta^{2} \psi   + \Psi_{2}\Delta \psi  + \sum_{i, j=1}^n  \Psi_{3} ^{ij} \psi_{x_{i}x_{j}}  + \sum_{i=1}^n \Psi_{4}^{i}  \psi_{x_{i}}    + \Psi_{5} \psi  ,
\end{equation*}
with 
\begin{equation*}
    \begin{cases}
        \Psi_{2} = 2 s^{2}\lambda^{2}\xi^{2} |\nabla\eta|^{2},
        \quad
        \Psi_{3}^{ij} = 4 s^{2}\lambda^{2}\xi^{2} \eta_{x_{i}} \eta_{x_{j}},
        \\ 
        \Psi_{4}^{i} = 12 s^{2}\lambda^{3}\xi^{2} |\nabla\eta|^{2} \eta_{x_{i}}  +4 s^{2} \lambda^{2} \xi^{2} \Delta \eta \eta_{x_{i}}  
        + \sum\limits_{j=1}^n 8 s^{2}\lambda^{2}\xi^{2} \eta_{x_{i}x_{j}} \eta_{x_{j}},
        \\  
        \Psi_{5} = s ^{4}\lambda ^{4}\xi^{4} |\nabla \eta|^{4},
    \end{cases}\quad i,j=1,\cdots,n.
    \end{equation*}
\begin{equation*}
 	K_{2}  =    \sum_{i=1}^n \Phi_{1}^{i} \Delta \psi_{x_{i}} d t  + \Phi_{2} \Delta \psi d t  + \sum_{i, j=1}^n \Phi_{3}^{ij}  \psi_{x_{i}x_{j}} d t  + \sum_{i=1}^n \Phi_{4} ^{i}  \psi_{x_{i}} d t  + \Phi_{5} \psi d t
    + \kappa d \psi 
 	,
\end{equation*}
with 
\begin{equation*}
	\begin{cases} \Phi_{1}^{i} = -4 s \lambda \xi  \eta_{x_{i}},
 	\quad
 	\Phi_{2} =  -2 s\lambda^{2}\xi |\nabla \eta|^{2} - 2s\lambda\xi  \Delta \eta -  \lambda,
 	\\ 
 	\Phi_{3}^{ij} = 4 s\lambda\xi  \eta_{x_{i}x_{j}} - 4 s\lambda^{2}\xi \eta_{x_{i}} \eta_{x_{j}},\quad \Phi_{4}^{i} = -4 s^{3} \lambda^{3} \xi^{3} |\nabla\eta|^{2} \eta_{x_{i}} ,
 	\\  \Phi_{5} =
 	- 6 s^{3} \lambda ^{4}\xi^{3} |\nabla\eta|^{4}
 	- 12 s ^{3}\lambda ^{3}\xi^{3}  (\nabla^{2}\eta \nabla \eta \nabla \eta )
 	-2 s^{3} \lambda ^{3}\xi^{3} |\nabla\eta|^{2} \Delta \eta
 	- s ^{3}\lambda ^{\frac{7}{2}}\xi^{3} |\nabla\eta|^{4},
    \end{cases} \!\!\! i,j=1,\cdots,n.
\end{equation*}
\begin{align*} 
	\notag
	K_{3} & =
    -8 s \lambda \xi \sum\limits_{i,j=1}^{n} \eta_{x_{i}x_{j}} \psi_{x_{i}x_{j}}dt
    -4\nabla  \Delta \ell \cdot \nabla \psi  d t + 4  (\nabla \ell \cdot \nabla \Delta \ell ) \psi d t 
    +2 |\nabla^{2} \ell|^{2} \psi d t-\Delta^{2} \ell \psi d t 
    \\ \notag
    & \quad
    +|\Delta \ell|^{2} \psi d t  
    + 8 s^{3} \lambda^{3} \xi^{3}(\nabla^{2} \eta \nabla \eta \nabla \eta) \psi d t
    -\kappa s \alpha_{t} \psi d t
    + s ^{3}\lambda ^{\frac{7}{2}}\xi^{3} |\nabla\eta|^{4} \psi dt
    + \lambda \Delta \psi dt
    ,
\end{align*}
\begin{align*}
    V_{1} & =( V_{1}^{1}, V_{1}^{2},\cdots, V_{1}^{n} ),
    \quad
    V_{2} = ( V_{2}^{1}, V_{2}^{2},\cdots, V_{2}^{n} ),
    \\
    V_{1}^{j} & =  \sum_{i,k=1}^n\Big[ \sum_{l=1}^n \Phi_{1}^{l} \psi_{x_{k}x_{k}x_{l}} \psi_{x_{i}x_{i}x_{j}}  -  \frac{1}{2} \sum_{l=1}^n  \Phi_{1}^{j} \psi_{x_{k}x_{k}x_{l}} \psi_{x_{i}x_{i}x_{l}}
        + \frac{1}{2}  \Psi_{2} \Phi_{1}^{j}\psi_{x_{i}x_{i}}\psi_{x_{k}x_{k}}
      \\
      &  \quad \quad \quad \ \
       + \sum_{l=1}^{n} \Psi_{3}^{ik}\Phi_{1}^{l} \psi_{x_{i}x_{k}}\psi_{x_{l}x_{j}}
       -  \frac{1}{2}  \sum_{l=1}^{n}\Psi_{3}^{ij}\Phi_{1}^{l} \psi_{x_{i}x_{k}}\psi_{x_{k}x_{l}}
        + \Phi_{4}^{k} \psi_{x_{i}x_{i}x_{j}} \psi_{x_{k}}
        \\
        & \quad \quad \quad \ \
        - \Phi_{4}^{k} \psi_{x_{i}x_{j}} \psi_{x_{i}x_{k}}
        + \frac{1}{2} \Phi_{4}^{j}\psi_{x_{i}x_{k}}^{2}
        +\Big ( \Psi_{2} \Phi_{4}^{k} \delta_{ij}
        - \frac{1}{2} \Psi_{2} \Phi_{4}^{j} \delta_{ik}
        - \frac{1}{2} \Psi_{5} \Phi_{1}^{j} \delta_{ik}
        + \Psi_{3}^{ij}\Phi_{4}^{k} \Big ) \psi_{x_{i}}\psi_{x_{k}}\Big],
    \\
    V_{2}^{j}  &=
    \sum_{i,l,r,m=1}^n \Theta_{1}^{ijlrm} \psi_{x_{i}x_{i}x_{l}}\psi_{x_{r}x_{m}}
    + \sum_{i,k,l,r=1}^n \Theta_{2}^{ijklr} \psi_{x_{i}x_{k}}\psi_{x_{l}x_{r}}
    + \sum_{i,k,l=1}^n \Theta_{3}^{ijkl} \psi_{x_{i}x_{k}}\psi_{x_{l}}
    \\
    & \quad
    + \sum_{i,k=1}^n \Theta_{4}^{ijk} \psi_{x_{i}}\psi_{x_{k}}
    + \sum_{i=1}^n \Theta_{5} \psi_{x_{i}x_{i}x_{j}}\psi
    + \sum_{i,k=1}^n \Theta_{6}^{ijk} \psi_{x_{i}x_{k}}\psi
    + \sum_{i=1}^n \Theta_{7}^{ij} \psi_{x_{i}}\psi
    +  \Theta_{8}^{j} \psi^{2}
,
\end{align*}
\begin{align*}
    M & =  
    - \sum_{i,j,k,l = 1}^{n} (
        \Phi_{1x_{j}}^{ l } 
        + \Phi_{3}^{j l } 
    )  \psi_{x_{k}x_{k}x_{ l }} \psi_{x_{i}x_{i}x_{j}}
    + \sum_{i = 1}^{n} \biggl( 
        \dfrac{1}{2} \Phi_{1 x_{i}}^{i} - \Phi_{2}
    \biggr) |\nabla \Delta \psi |^{2}
    \\
    & \quad 
    + \sum_{i,j,k,l = 1}^{n} [
        (- \Psi_{3}^{i j} \Phi_{1}^{l} )_{x_{k}}
        + \Psi_{3}^{i l} \Phi_{3}^{k j}
    ] \psi_{x_{i}x_{j}} \psi_{x_{k}x_{l}}
    \\
    & \quad 
    + \sum_{i,j,k = 1}^{n} \bigg[
        2 \Phi_{4 x_{j}}^{k}
        + \Psi_{2} \Phi_{3}^{j k}
        + \dfrac{1}{2} \sum_{l=1}^{n} (\Psi_{3}^{l j} \Phi_{1}^{k})_{x_{l}}
        + \Psi_{3}^{j k} \Phi_{2}
        - \Psi_{4}^{j} \Phi_{1}^{k}
    \bigg] 
    \psi_{x_{i}x_{j}} \psi_{x_{i}x_{k}}
    \\
    & \quad
    - \dfrac{1}{2} \sum_{i=1}^{n} \Phi_{4 x_{i} }^{i} |\nabla^{2} \psi |^{2}
    + \bigg [
        \Phi_{5} 
        -  \dfrac{1}{2} \sum_{i=1}^{n}  (\Psi_{2} \Phi_{1}^{i})_{x_{i}}
        + \Phi_{2} \Psi_{2}
    \bigg ] |\Delta \psi|^{2}
    \\
    & \quad
    + \sum_{i, j =1}^{n} \biggl[
        - (\Psi_{2} \Phi_{4}^{i})_{x_{j}} 
        - \dfrac{1}{2} \sum_{k=1}^{n} ( \Psi_{3}^{i k} \Phi_{4}^{j} ) _{x_{k}}
        - \Psi_{3}^{i j} \Phi_{5}
        + \Psi_{4}^{i} \Phi_{4}^{j}
        + (\Psi_{5} \Phi_{1}^{i})_{x_{j}}
        - \Psi_{5}\Phi_{3}^{i j}
    \biggr] \psi_{x_{i}} \psi_{x_{j}}
    \\
    & \quad 
    + \biggl[
        - \dfrac{1}{2} \sum_{i=1}^{n}  (\Psi_{5} \Phi_{4}^{i})_{x_{i}} 
        + \Psi_{5} \Phi_{5}
    \biggr] | \psi|^{2}
    + \biggl[
        \dfrac{1}{2} \sum_{i=1}^{n} (\Psi_{2} \Phi_{4}^{i} )_{x_{i}}
        - \Psi_{2} \Phi_{5}
        + \dfrac{1}{2} \sum_{i=1}^{n} (\Psi_{5} \Phi_{1}^{i} )_{x_{i}}
        - \Psi_{5} \Phi_{2}
    \biggr] |\nabla \psi|^{2}
    ,
\end{align*}
\begin{align*}
    & \Lambda_{1}^{ijkl} =
    \sum_{r=1}^{n} \Big(
        \frac{1}{2} \Phi_{2x_{r}x_{r}} \delta_{ij}\delta_{kl}
        - \Phi_{3x_{j}x_{r}}^{kr} \delta_{il}
        - \Phi_{3x_{i}x_{r}}^{kr} \delta_{lj}
        +\frac{1}{2}\sum_{m=1}^{n} \Phi_{3x_{r}x_{m}}^{rm} \delta_{ik} \delta_{lj}
    \Big)
    + \Phi_{3x_{i}x_{j}}^{kl}
    + \Phi_{3x_{i}x_{l}}^{kj}
    ,
    \\
   & \Lambda_{2}^{ij}
        =
        \sum_{k=1}^n\Big [  - \Phi_{4 x_{i}x_{k}x_{k}}^{j}
     +  \frac{1}{2} \Phi_{4x_{i}x_{j}x_{k}}^{k}
     - 2\Phi_{5x_{i}x_{j}}
    + (\Psi_{2}\Phi_{3}^{jk} )_{x_{i}x_{k}}
    - \frac{1}{2}  (\Psi_{2}\Phi_{3}^{ij} )_{x_{k}x_{k}}
    - \sum_{l=1}^{n} \frac{1}{2}   (\Psi_{2}\Phi_{3}^{kl} \delta_{ij} )_{x_{k}x_{l}}
    \\
    & \quad \quad \quad \quad \ 
    +   (\Psi_{3}^{ik}\Phi_{2}  )_{x_{k}x_{j}}
    - \sum_{l=1}^{n} \frac{1}{2}  (\Psi_{3}^{kl}\Phi_{2} \delta_{ij} )_{x_{k}x_{l}}
    - \frac{1}{2}  (\Psi_{3}^{ij}\Phi_{2}  )_{x_{k}x_{k}}
    + \sum_{l=1}^{n}  (\Psi_{3}^{ik}\Phi_{3}^{jl} )_{x_{k}x_{l}}
    \\
    & \quad \quad \quad \quad \ 
    - \sum_{l=1}^{n} \frac{1}{2}  (\Psi_{3}^{ij}\Phi_{3}^{kl} )_{x_{k}x_{l}}
    - \frac{1}{2}  (\Psi_{3}^{kl}\Phi_{3}^{ij} )_{x_{k}x_{l}}
    + \frac{1}{2}  (\Psi_{4}^{i}\Phi_{1}^{j} )_{x_{k}x_{k}}
    +  (\Psi_{4}^{i}\Phi_{2} )_{x_{j}}
    + \frac{1}{2}  (\Psi_{4}^{k}\Phi_{2}  \delta_{ij} )_{x_{k}}
    \\
    & \quad \quad \quad \quad \ 
    -  (\Psi_{4}^{i}\Phi_{3}^{jk} )_{x_{k}}
    + \frac{1}{2}  (\Psi_{4}^{k}\Phi_{3}^{ij} )_{x_{k}}
    \Big ],
    \\
      &   \Lambda_{3} =
        \sum_{i=1}^n \Big[  \sum_{j=1}^n \frac{1}{2} \Phi_{5x_{i}x_{i}x_{j}x_{j}}
        + \sum_{j=1}^n \frac{1}{2}   ( \Psi_{2} \Phi_{5}  )_{x_{j}x_{j}}
        + \sum_{j=1}^n \frac{1}{2} (\Psi_{3}^{ij} \Phi_{5})_{x_{i}x_{j}}
        - \frac{1}{2}   ( \Psi_{4} ^{i} \Phi_{5}  )_{x_{i}}
        \\
        &  \quad \quad \quad \quad \
        - \sum_{j=1}^n \frac{1}{2} ( \Psi_{5} \Phi_{1}^{i}  )_{x_{i}x_{j}x_{j}}
        + \frac{1}{2}  ( \Psi_{5} \Phi_{2}  )_{x_{i}x_{i}}
         \Big]
        ,
    \end{align*}
    and 
    \begin{align*}
     &   \Theta_{1}^{ijlrm} =
       \Phi_{2} \delta_{lj}\delta_{rm}
     + \Phi_{3}^{rm}    \delta_{lj}
     - \Phi_{3}^{mj}    \delta_{lr}
     + \Phi_{3}^{lr}    \delta_{mj},                                                \\
     &   \Theta_{2}^{ijklr} =
        -\frac{1}{2}\Phi_{2x_{j}} \delta_{ik}\delta_{lr}
        - \Phi_{3x_{i}}^{lr} \delta_{kj}
        + \Phi_{3x_{k}}^{rj} \delta_{il}
        + \sum_{m=1}^{n} \Phi_{3x_{m}}^{rm} \delta_{kj}\delta_{il}
        - \frac{1}{2} \sum_{m=1}^{n} \Phi_{3x_{m}}^{jm} \delta_{il} \delta_{kr}
        - \Phi_{3x_{i}}^{kl} \delta_{rj}
        ,
        \\
      &   \Theta_{3}^{ijkl} =
        - \Phi_{4x_{i}}^{l}\delta_{kj}
        - \Phi_{5}\delta_{ik}\delta_{lj}
        + \Psi_{2}\Phi_{3}^{ik}\delta_{lj}
        -\Psi_{2}\Phi_{3}^{ij}\delta_{kl}
        + \Psi_{3}\Phi_{2}\delta_{ik}
        -\Psi_{3}\Phi_{2}\delta_{ij}
        +\Psi_{3}^{ik}\Phi_{3}^{lj}
        \\
        & \qquad\quad\;
        -\Psi_{3}^{ij}\Phi_{3}^{kl}
        +\Psi_{4}^{l}\Phi_{1}^{k}\delta_{ij}
        \\
       &   \Theta_{4}^{ijk}=
        \Phi_{4x_{i}x_{j}}^{k}
         -  \frac{1}{2}\Phi_{4x_{i}x_{k}}^{j}
         + 2 \Phi_{5x_{i}}\delta_{kj}-\Phi_{5x_{j}}\delta_{ik}
         -  ( \Psi_{2}\Phi_{3}^{jk}  )_{x_{i}}
         + \frac{1}{2}  ( \Psi_{2}\Phi_{3}^{ik}  )_{x_{j}}
         + \sum_{l=1}^{n} \frac{1}{2}  ( \Psi_{2}\Phi_{3}^{jl})_{x_{l}} \delta_{ik}
        \\
        & \qquad\quad 
        - \sum_{l=1}^{n}  ( \Psi_{3}^{il}\Phi_{2}  )_{x_{l}}\delta_{kj}
        +  \frac{1}{2} \sum_{l=1}^{n}  ( \Psi_{3}^{lj}\Phi_{2}  )_{x_{l}}\delta_{ik}
         +  \frac{1}{2}   ( \Psi_{3}^{ik}\Phi_{2}  )_{x_{j}}
         -  \sum_{l=1}^{n}  ( \Psi_{3}^{ij}\Phi_{3}^{kl}  )_{x_{l}}
         + \frac{1}{2}  \sum_{l=1}^{n}  ( \Psi_{3}^{ik}\Phi_{3}^{jl}  )_{x_{l}}
        \\
        & \qquad\quad 
        + \frac{1}{2}  \sum_{l=1}^{n}  ( \Psi_{3}^{lj}\Phi_{3}^{ik}  )_{x_{l}}
        - \frac{1}{2}    ( \Psi_{4}^{i}\Phi_{1}^{k}  )_{x_{j}}
        +  \Psi_{4}^{i} \Phi_{2} \delta_{kj} 
        - \frac{1}{2}\Psi_{4}^{j} \Phi_{2} \delta_{ik}
        + \Psi_{4}^{i}\Phi_{3}^{kj}
        -\frac{1}{2} \Psi_{4}^{j}\Phi_{3}^{ik}
    \\
      &   \Theta_{5}= \Phi_{5},
        \\
      &   \Theta_{6}^{ijk}= - \Phi_{5x_{j}}\delta_{ik}+ \Psi_{5}\Phi_{1}^{k}\delta_{ij},
        \\
      &   \Theta_{7}^{ij} = \sum_{k=1}^{n} \Phi_{5x_{k}x_{k}}\delta_{ij} + \Psi_{2} \Phi_{5} \delta_{ij} + \Psi_{3}^{ij}\Phi_{5} -  ( \Psi_{5}\Phi_{1}^{i}  )_{x_{j}} + \Psi_{5}\Phi_{2} \delta_{ij} + \Psi_{5}\Phi_{3}^{ij},
        \\
      &   \Theta_{8}^{j} = - \frac{1}{2} \sum_{i=1}^n \Phi_{5x_{i}x_{i}x_{j}} - \frac{1}{2} ( \Psi_{2}\Phi_{5}  )_{x_{j}} - \frac{1}{2} \sum_{i=1}^n  ( \Psi_{3}^{ij} \Phi_{5}  )_{x_{i}} + \frac{1}{2} \Psi_{4}^{j}\Phi_{5}   
      +\frac{1}{2}\sum_{i=1}^n  ( \Psi_{5}\Phi_{1}^{j}  )_{x_{i}x_{i}}
      -\frac{1}{2}    (\Psi_{5}\Phi_{2} )_{x_{j}}
      \\
      & \quad \quad \ 
        - \frac{1}{2}\sum_{i=1}^n  ( \Psi_{5}\Phi_{3}^{ij}  )_{x_{i}}
        +\frac{1}{2} \Psi_{5} \Phi_{4}^{j}
        .
    \end{align*}
\end{theorem}

\Cref{thm.fundamentalIndentity} can be proved by following the steps of the proof of the fundamental identity presented in \cite[Theorem 2.2]{Wang}. So we omit the proof here.

It is easy to see that for all $ t \in (0,T) $ and $ x \in G $,
\begin{equation}\label{eqEstimatesOrder}
    |\xi_{t}|  \leq \frac{T}{2} \xi^{3}, \quad |\alpha_{t}|  \leq \frac{T}{2} \xi^{3}, \quad \xi^{-1} \leq \frac{T}{2}.
\end{equation}

In what follows, for a nonnegative integer $m$, we denote by $O(\lambda^{m})$ a function of order $\lambda^{m}$ for large $\lambda$.
Similarly to \cite[Proposition 2.2]{Wang}, we have the following estimate.

\begin{proposition}\label{prop.estimatesOrder}
    When $ s $ and $ \lambda $ are large enough, it holds that\vspace{-1mm}
    \begin{equation}\label{eq.estimatesOrderInequalities}
        \begin{cases}
            \begin{alignedat}{2}
                &\sum_{i,j,k,l=1}^{n} \Lambda^{ijkl}_{1}  \zeta _{ij} \zeta_{kl} \geq  s  \xi O(\lambda^{4}) |\zeta|^{2}, && \quad \forall \zeta=( \zeta_{ij} ) \in \mathbb{R}^{n\times n},
                \\
                &\sum_{i, j=1}^n \Lambda_{2}^{ij} \zeta_{i} \zeta_{j} \geq \big[s^{3}\xi^{3}O(\lambda^{6}) + Ts^{2}\xi^{4}O(\lambda^{2}) \big] |\zeta|^{2}  , && \quad \forall \zeta= ( \zeta_{1},\dots,\zeta_{n} ) \in \mathbb{R}^{n},
                \\
                &\Lambda_{3} \geq s ^{5} \xi^{5} O(\lambda^{8})  + Ts^{4}\xi^{6}O(\lambda^{4}) . &&
            \end{alignedat}
        \end{cases}\vspace{-1mm}
    \end{equation}
    for any $ t \in (0,T) $ and $ x \in G $.
\end{proposition}

From the Carleman inequality for the Laplace operator with homogeneous Dirichlet boundary conditions (see \cite[Lemma 2.4]{Guerrero2019}), we obtain the following result.

\begin{lemma}\label{lemma.LaplaceCarleman}\cite[Lemma 2.3]{Wang}
    Let $ q \in H^{2}(G)\cap H_0^1(G) $ and $ \eta $ be given by \eqref{eq.eta}. Then there exists a constant $C>0$ independent of $s$ and $\lambda$, and parameters $\hat{\lambda}>1$ and $\hat{s}>1$ such that for all $\lambda \geq \hat{\lambda}$ and for all $ s  \geq \hat{s}  $, 
    \begin{align} \label{eq.LaplaceCarleman} \notag
        & \int_{Q}  s^{6} \lambda^{8} \xi^{6}\theta^{2} |q|^{2} d x d t\!+\!  \int_{Q}   s^{4} \lambda^{6}  \xi^{4}\theta^{2} |\nabla q|^{2} d x d t
        \\
        & 
        \leq  C  \Big(  \int_{(0, T)\times G_{2}}   s^{6} \lambda^{8} \xi^{6}\theta^{2} |q|^{2} d x d t+   \int_{Q}   s^{3} \lambda^{4}  \xi^{3}\theta^{2} |\Delta q|^{2} d x d t \Big
        ).
    \end{align}
\end{lemma}

Let
\begin{equation}
    \label{eq.alphastar}
\alpha_{\star}(t)=\frac{e^{\lambda|2 \eta|_{C(\overline{G})}}-e^{4 \lambda|\eta|_{C(\overline{G})}}}{t^{1/2}(T-t)^{1/2}}, \quad \quad
\xi_{\star}(t)=\frac{e^{\lambda|2 \eta|_{C(\overline{G})}}}{t^{1/2}(T-t)^{1/2}}.
\end{equation}

\begin{lemma}\label{lemma.auxiliaryEstimates}
    There exists a constant $ C>0 $, such that for 
    $ \varphi \in  L_{\mathbb{F}}^{2}(0, T ; H^{4}(G)) , f \in L_{\mathbb{F}}^{2}(0, T ; L^{2}(G))$ and $ g \in  L_{\mathbb{F}}^{2}(0, T ; H^{2}(G)) $  satisfying 
    \begin{equation}
        \label{eqauxiliaryEstimatesEq1}
        \left\{
        \begin{alignedat}{2}
            &  d \varphi  + \Delta^2 \varphi  d t=    f d t+g d W(t) && \quad \text { in } Q, \\
            & \varphi = \frac{\partial \varphi}{\partial \nu} =0 && \quad \text { on } \Sigma,
        \end{alignedat}
        \right. 
    \end{equation}
    for $ s\geq C T^{1/2} $, it holds that
    \begin{align*}
        \notag
        &\mathbb{E} \int_{\Sigma} \Big[
        s^{\frac{9}{4}} \lambda^{3} \xi^{\frac{9}{4}} \theta^{2} |\nabla^{2} \varphi|^{2}
        + s^{\frac{3}{4}} \lambda     \xi^{\frac{3}{4}} \theta^{2} |\nabla^{3} \varphi|^{2}
        \Big]d \Gamma dt
        + \mathbb{E} \int_{0}^{T} e^{2s\alpha_{\star}} |\varphi|^{2}_{H^{4}(G)} dt
        \\
        &
        \leq
        C \mathbb{E} \int_{Q}  s^{6} \lambda^{8} \xi^{6} \theta^{2} |\varphi|^{2} dxdt
        +  C  \mathbb{E} \int_{Q}  \theta^{2}
        (
            |f|^{2} 
            +  | \Delta g| ^{2}
            + s^{2} \lambda^{2} \xi^{2}|\nabla g|  ^{2}
            +  s^{4} \lambda^{4} \xi^{4} g ^{2}
        )d x d t
        .
    \end{align*}
\end{lemma}

The proof of \cref{lemma.auxiliaryEstimates} is similar to that of \cite[Lemma 2.4]{Wang}, which is put in \cref{secProofLemma} for the sake of completeness.

\section{Proof of Carleman estimate} \label{sc.3}

\begin{proof}[Proof of \cref{thm.carlemanForwardPre}]

By a density argument,  we assume $ \varphi \in L_{\mathbb{F}}^{2}(0, T ; H^{4}(G)) $.

In order to shorten the formulas used in the sequel, we define the lower order terms  
\begin{align}
    \notag
    &\mathcal{A}_{1} \deq s^{5}  \mathbb{E}\int_{Q} \xi^{5}[O(\lambda^{8})+s  \xi O(\lambda^{7})]|\psi|^{2} d x d t ,
    & &
     \mathcal{A}_{2} \deq s^{3}  \mathbb{E}\int_{Q} \xi^{3}[O( \lambda^{6} )+s \xi O( \lambda^{5} )]|\nabla \psi|^{2} d x d t,
    \\ \label{eq.pfthm2.1.S8}
    &\mathcal{A}_{3} \deq s  \mathbb{E}\int_{Q} \xi[O( \lambda^{4} )+s \xi O( \lambda^{3} )]|\nabla^{2} \psi|^{2} d x d t,
    & &
    \mathcal{A}_{4} \deq    \mathbb{E}\int_{Q} O(1 ) | \nabla \Delta \psi|^{2} d x d t,
     \\ \notag
    & 
    \mathcal{B}_{1} \deq  s^{2} \mathbb{E} \int_{\Sigma} \xi^{2} O ( \lambda^{3} ) \theta^{2}|\nabla^{2} \varphi|^{2} d \Gamma d t,
    & &
    \mathcal{B}_{2} \deq  s^{1/2}  \mathbb{E} \int_{\Sigma} \xi^{1/2} O ( \lambda ) \theta^{2}|\nabla^{3} \varphi|^{2} d \Gamma d t,
\end{align}
and
\begin{equation*}
    \mathcal{A}=\mathcal{A}_{1}+\mathcal{A}_{2}+\mathcal{A}_{3}  +\mathcal{A}_{4},
    \quad \quad
    \mathcal{B}=\mathcal{B}_{1}+\mathcal{B}_{2}.
\end{equation*}
Integrating the equality \cref{eq.finalEquation} with $ \kappa = 1 $ on $Q$, taking mathematical expectation in both sides, and noting \cref{eqEstimatesOrder}, \cref{eq.estimatesOrderInequalities}, we conclude that, there exists a $ s_{1} \geq C T^{1/2} $, such that for $ s \geq s_{1} $, it holds that 
\begin{align}\label{eq.thm31S1}
    \notag
    & 2 \mathbb{E}\int_{Q} \theta K_{1} (  d \varphi + \Delta^{2} \varphi dt ) d x
    - 2  \mathbb{E}\int_{Q} \operatorname{div} ( V_{1} + V_{2} )dxdt
    \\ \notag
    & 
    - 2 \mathbb{E}\int_{Q} \sum_{i, j=1}^n (
   \psi_{x_{i}x_{i}x_{j}} d \psi - \psi_{x_{i}x_{j}}d \psi_{x_{i}}
   +\Psi_{2} \psi_{x_{i}} \delta_{ij} d \psi
   +\Psi_{3}^{ij} \psi_{x_{i}} d \psi
   )_{x_{j}} d x
   \\ \notag
    & \geq
   2 \mathbb{E}\int_{Q} K_{1}^{2} dxdt
   + 2 \mathbb{E}\int_{Q} K_{1} K_{3} d x
   + 2 \mathbb{E}\int_{Q} M dxdt
   \\ \notag
    & \quad
   - \mathbb{E}\int_{Q} \theta^{2} \sum_{i, j=1}^n\big[  d \varphi_{x_{i}x_{j}} + 2 \ell_{x_{j}} d \varphi_{x_{i}} +   ( \ell_{x_{i}x_{j}} + \ell_{x_{i}} \ell_{x_{j}} ) d \varphi \big]^{2} d x
   + \mathbb{E}\int_{Q} \theta^{2}  \Psi_{2} | d \nabla \varphi     + \nabla \ell d \varphi|^{2} d x
    \\ 
    & \quad
    + 4 \mathbb{E}\int_{Q}  s^{2}\lambda^{2}\xi^{2} \theta^{2} | \nabla \eta d \nabla \varphi
    + \nabla \eta \nabla \ell d \varphi|^{2} d x
    -  \mathbb{E}\int_{Q}  \Psi_{6} \theta^{2} ( d \varphi )^{2} d x + \mathcal{A}
   .
\end{align}

We will estimate the divergence terms, the interior terms  and the stochastic terms in the following Steps 1--3.

{\bf Step 1}. In this step, we deal with the divergence terms.

Firstly, we estimate $ V_{1} $. Since $ \varphi = \frac{\partial \varphi}{\partial \nu} = 0 $  and $ \psi = \theta \varphi $, we have $ \psi = 0 $ and  $\nabla \psi = 0 $ on  the boundary $ \Sigma $.
Hence, on  $ \Sigma $, we obtain 
\begin{align} \label{eqCarlemanBoundary1}
    \notag
    V_{1} \cdot \nu 
    & =  \sum_{i,j,k=1}^n\Big[ 
    \sum_{l=1}^n \Phi_{1}^{l} \psi_{x_{k}x_{k}x_{l}} \psi_{x_{i}x_{i}x_{j}}  
    -  \frac{1}{2} \sum_{l=1}^n  \Phi_{1}^{j} \psi_{x_{k}x_{k}x_{l}} \psi_{x_{i}x_{i}x_{l}}
    + \frac{1}{2}  \Psi_{2} \Phi_{1}^{j}\psi_{x_{i}x_{i}}\psi_{x_{k}x_{k}}
    \\ \notag
    &  \quad \quad \quad \ \
    + \sum_{l=1}^{n} \Psi_{3}^{ik}\Phi_{1}^{l} \psi_{x_{i}x_{k}}\psi_{x_{l}x_{j}}
    -  \frac{1}{2}  \sum_{l=1}^{n}\Psi_{3}^{ij}\Phi_{1}^{l} \psi_{x_{i}x_{k}}\psi_{x_{k}x_{l}}
    - \Phi_{4}^{k} \psi_{x_{i}x_{j}} \psi_{x_{i}x_{k}}
    + \frac{1}{2} \Phi_{4}^{j}\psi_{x_{i}x_{k}}^{2}
    \Big] \nu^{j},
    \\ \notag
    & =
    - 4 \sum_{i,j,k,p = 1}^{n}  s \lambda \xi  \eta_{x_{p}} \psi_{x_{k}x_{k}x_{p}} \psi_{x_{i}x_{i}x_{j}} \nu^{j}
    +2  \sum_{i,k,p=1}^{n}  s \lambda \xi  \frac{\partial \eta}{\partial \nu} \psi_{x_{k}x_{k}x_{p}} \psi_{x_{i}x_{i}x_{p}} 
    \\ \notag
    & \quad 
    - 4 \sum_{i,k=1}^{n}    s^{3} \lambda^{3} \xi^{3} \bigg(\frac{\partial \eta}{\partial \nu}\bigg)^{3}   \psi_{x_{i}x_{i}} \psi_{x_{k}x_{k}}   
    -16  \sum_{i,j,k,p=1}^{n}    s^{3} \lambda^{3} \xi^{3} \bigg(\frac{\partial \eta}{\partial \nu}\bigg)^{3}  \psi_{x_{i}x_{k}} \nu^{i} \nu^{k} \psi_{x_{p}x_{j}} \nu^{p} \nu^{j}
    \\
    & \quad 
    + 12 \sum_{i,j,k =1}^{n}  s^{3} \lambda^{3} \xi^{3} \bigg(\frac{\partial \eta}{\partial \nu}\bigg)^{3}  \psi_{x_{i}x_{j}}  \nu^{j} \psi_{x_{i}x_{k}} \nu^{k}
    -2 \sum_{i,k=1}^{n}  s^{3} \lambda^{3} \xi^{3} \bigg(\frac{\partial \eta}{\partial \nu}\bigg)^{3} \psi_{x_{i}x_{k}}^{2} 
    .
\end{align}
Since $ \nabla \psi = 0 $ on $ \Sigma $, it holds that 
\begin{align}
    \label{eqCarlemanBoundary2}
    \sum_{i,j=1}^{n} \psi_{x_{i}x_{j}} \nu^{i} \nu^{j} 
    = \sum_{i,j=1}^{n} \left( \psi_{x_{i}} \nu^{i} \nu^{j} \right)_{x_{j}} - \sum_{i,j=1}^{n}  \psi_{x_{i}} \left( \nu^{i} \nu^{j} \right)_{x_{j}}
    = \sum_{j=1}^{n} \bigg( \frac{\partial \psi}{\partial \nu} \nu^{j} \bigg)_{x_{j}}
    = \sum_{j=1}^{n} \psi_{x_{j}x_{j}},
\end{align}
and 
\begin{align}
    \label{eqCarlemanBoundary3}
    \notag
    \sum _{i,j=1}^{n} \psi_{x_{i}x_{j}}^{2} 
    & = \sum_{i=1}^{n} \Big( \sum_{j=1}^{n} \psi_{x_{i}x_{j}}^{2}  \Big)
    = \sum_{i=1}^{n} \left( |\nabla \psi_{x_{i}}|^{2}  \right)
    = \sum_{i=1}^{n} \left( |\nabla \psi_{x_{i}} \cdot \nu|^{2}  \right)
    \\
    &
    = \sum_{i=1}^{n} \Big( \sum_{j=1}^{n} \psi_{x_{i}x_{j}} \nu^{j}  \Big)^{2}
    = \sum_{i,j,k=1}^{n} \psi_{x_{i}x_{j}} \nu^{j} \psi_{x_{i}x_{k}} \nu^{k}
    .
\end{align}
Combining \cref{eqCarlemanBoundary1,eqCarlemanBoundary2,eqCarlemanBoundary3}, on the boundary $ \Sigma $, we deduce that 
\begin{align} \label{eqCarlemanBoundary4} 
    \notag
    V_{1} \cdot \nu  
    &= 
    -4 \sum_{i,j,k,p=1}^{n}  s \lambda \xi  \eta_{x_{p}} \psi_{x_{k}x_{k}x_{p}} \psi_{x_{i}x_{i}x_{j}} \nu^{j}
    + 2 \sum_{i,k,p=1}^{n}   s \lambda \xi  \frac{\partial \eta}{\partial \nu} \psi_{x_{k}x_{k}x_{p}} \psi_{x_{i}x_{i}x_{p}} 
    \\
    & \quad 
    -20 \sum_{i,k=1}^{n}    s^{3} \lambda^{3} \xi^{3} \bigg(\frac{\partial \eta}{\partial \nu}\bigg)^{3}   \psi_{x_{i}x_{i}} \psi_{x_{k}x_{k}}   
    + 10  \sum_{i,j=1}^{n}  s^{3} \lambda^{3} \xi^{3} \bigg(\frac{\partial \eta}{\partial \nu}\bigg)^{3}  \psi_{x_{i}x_{j}}^{2} .
\end{align}

We claim that, on the boundary $ \Sigma $,
\begin{align} 
    \label{eqCarlemanBoundary5}
    \sum_{i,j=1}^{n} \psi_{x_{i}x_{i}} \psi_{x_{j}x_{j}} = \sum_{i,j=1}^{n} \psi_{x_{i}x_{j}}^{2}.
\end{align}
In fact, for $ i =1 , \cdots, n $, since $ \psi_{x_{i}}=0 $ on $ \Sigma $, we can write
$ \nabla \psi_{x_{i}} = \mathscr{K}_{i} \nu $,
where $ \mathscr{K}_{i}=\frac{\partial \psi_{i}}{\partial \nu}$, which means $ \psi_{x_{i}x_{j}} = \mathscr{K}_{i} \nu^{j} $ for all $ j =1 , \cdots, n $. 
Notice that, for $ i,j= 1 , \cdots, n $,  $ \psi_{x_{i}x_{j}}=\psi_{x_{j}x_{i}} $.
We obtain $ \mathscr{K}_{i}\nu^{j}=\mathscr{K}_{j} \nu^{i} $.
Hence, it holds that 
\begin{align*}
    \notag
    \sum_{i,j=1}^{n} \psi_{x_{i}x_{i}} \psi_{x_{j}x_{j}}
    &= \sum_{i,j=1}^{n} \mathscr{K}_{i}\nu^{i}\mathscr{K}_{j}\nu^{j}
    = \sum_{i,j=1}^{n} \mathscr{K}_{i}\nu^{j}\mathscr{K}_{j}\nu^{i}
    = \sum_{i,j=1}^{n} \left( \mathscr{K}_{i}\nu^{j} \right)^{2}
     = \sum_{i,j=1}^{n} \psi_{x_{i}x_{j}}^{2}.
\end{align*} 
Combining \cref{eqCarlemanBoundary4,eqCarlemanBoundary5}, we obtain 
\begin{align} \label{eqCarlemanBoundary6}
    \notag
    V_{1} \cdot \nu  
    &= 
    -4 \sum_{i,j,k,p=1}^{n}  s \lambda \xi  \eta_{x_{p}} \psi_{x_{k}x_{k}x_{p}} \psi_{x_{i}x_{i}x_{j}} \nu^{j}
    + 2 \sum_{i,k,p=1}^{n}   s \lambda \xi  \frac{\partial \eta}{\partial \nu} \psi_{x_{k}x_{k}x_{p}} \psi_{x_{i}x_{i}x_{p}} 
    \\ \notag
    & \quad 
    -10 \sum_{i,k=1}^{n}    s^{3} \lambda^{3} \xi^{3} \bigg(\frac{\partial \eta}{\partial \nu}\bigg)^{3}   \psi_{x_{i}x_{i}} \psi_{x_{k}x_{k}}   
    \\
    & =-4 s \lambda \xi  \frac{\partial \eta}{\partial \nu}   |\nabla \Delta \psi \cdot \nu|^{2}
    + 2  s \lambda \xi  \frac{\partial \eta}{\partial \nu} |\nabla \Delta \psi|^{2} 
    -10  s^{3} \lambda^{3} \xi^{3} \bigg(\frac{\partial \eta}{\partial \nu}\bigg)^{3}  |\Delta \psi|^{2}
    .
\end{align}

Next, we consider the second term in the right hand side of \cref{eqCarlemanBoundary6}.
Put $ \tau \deq \frac{\nabla \Delta \psi - (\nabla \Delta \psi \cdot \nu)\nu}{|\nabla \Delta \psi - (\nabla \Delta \psi \cdot \nu)\nu|}. $ 
We have 
\begin{align}
    \label{eqCarlemanBoundary7}
    2 s \lambda \xi  \frac{\partial \eta}{\partial \nu} |\nabla \Delta \psi|^{2}
     = 2 s \lambda \xi  \frac{\partial \eta}{\partial \nu} \left( |\nabla \Delta \psi \cdot \nu|^{2} + |\nabla \Delta \psi \cdot \tau|^{2} \right).
\end{align}
Using the interpolation inequality, we obtain 
\begin{equation*}
    |\Delta \psi|_{H^{1}(\Gamma)} \leq |\Delta \psi|_{L^{2}(\Gamma)}^{1/3} |\Delta \psi|_{H^{3/2}(\Gamma)}^{2/3},
\end{equation*}
which implies 
\begin{align}
    \label{eqCarlemanBoundary8}
    \notag
    &\mathbb{E} \int_{\Sigma} 2 s \lambda \xi  \frac{\partial \eta}{\partial \nu}  |\nabla \Delta \psi \cdot \tau|^{2} d\Gamma dt
    \\\notag
    & \geq - C \mathbb{E} \int_{0}^{T} s \lambda \xi |\Delta \psi|_{H^{1}(\Gamma)}^{2} dt
    \\\notag
    & \geq - C \mathbb{E} \int_{0}^{T} s \lambda \xi |\Delta \psi|_{L^{2}(\Gamma)}^{2/3} |\Delta \psi|_{H^{3/2}(\Gamma)}^{4/3}  dt
    \\\notag
    & \geq \mathbb{E} \int_{\Sigma} s^{3} \lambda^{3} \xi^{3} \bigg( \frac{\partial \eta}{\partial \nu} \bigg)^{3} |\Delta \psi|^{2} d\Gamma dt
    - C  \mathbb{E} \int_{0}^{T} e^{2s\alpha_{\star}} |\varphi|_{H^{4}( G)}^{2} dt
    \\
    & = \mathbb{E} \int_{\Sigma} s^{3} \lambda^{3} \xi^{3} \bigg( \frac{\partial \eta}{\partial \nu} \bigg)^{3} |\nabla^{2} \psi|^{2} d\Gamma dt
    - C  \mathbb{E} \int_{0}^{T} e^{2s\alpha_{\star}} |\varphi|_{H^{4}( G)}^{2} dt
    ,
\end{align}
where we use the the fact that $  \frac{\partial \eta}{\partial \nu} \leq - C $ on $ \Sigma $ and \cref{eqCarlemanBoundary5}.
Combining \cref{eqCarlemanBoundary6,eqCarlemanBoundary7,eqCarlemanBoundary8}, we obtain 
\begin{align}
    \label{eqCarlemanBoundary9}\notag
    & \mathbb{E}\int_{Q} \operatorname{div}  V_{1} dxdt
    =\mathbb{E} \int_{\Sigma}  V_{1} \cdot \nu d\Gamma dt
    \\\notag
    & \geq \mathbb{E} \int_{\Sigma} \bigg[  
        - 2  s \lambda \xi  \frac{\partial \eta}{\partial \nu} |\nabla \Delta w \cdot \nu|^{2} 
        - 9 s^{3} \lambda^{3} \xi^{3} \bigg(\frac{\partial \eta}{\partial \nu}\bigg)^{3}  |\nabla^{2} w|^{2} 
    \bigg] d\Gamma dt 
    - C  \mathbb{E} \int_{0}^{T} e^{2s\alpha_{\star}} |\varphi|_{H^{4}(G)}^{2} dt
    \\
    & \geq 
    - C  \mathbb{E} \int_{0}^{T} e^{2s\alpha_{\star}} |\varphi|_{H^{4}(G)}^{2} dt
    .
\end{align}

Now, we estimate $ V_{2} $.
From \cref{eq.pfthm2.1.S8,eq.finalEquation} and $ \psi = |\nabla \psi | = 0 $ on the boundary, there exists a constant $ s_{2} \geq C T$, such that  for $ s \geq s_{2} $, we deduce
\begin{align}
    \label{eqCarlemanBoundary10}\notag
    & \mathbb{E}\int_{Q} \operatorname{div}  V_{2} dxdt = \mathbb{E} \int_{\Sigma}  V_{2} \cdot \nu d \Gamma dt
    \\\notag
    & =  \mathbb{E} \int_{\Sigma} \sum_{j=1}^{n} \Bigl[  
        \sum_{i,l,r,m=1}^n \Theta_{1}^{ijlrm} \psi_{x_{i}x_{i}x_{l}}\psi_{x_{r}x_{m}}
    + \sum_{i,k,l,r=1}^n \Theta_{2}^{ijklr} \psi_{x_{i}x_{k}}\psi_{x_{l}x_{r}} \Bigr] \nu^{j}d \Gamma dt 
    \\\notag
    & \geq 
    - C \mathbb{E} \int_{\Sigma}  \sum_{i,k,p,r,q=1}^{n} s \lambda^{2} \xi | \psi_{x_{i}x_{k}x_{p}}\psi_{x_{r}x_{q}} | d \Gamma dt 
    -  C \mathbb{E} \int_{\Sigma}  s \lambda^{3} \xi  |\nabla^{2} \psi|^{2}  d \Gamma dt 
    \\\notag
    & \geq
    - C \mathbb{E} \int_{\Sigma}  \bigl[  s^{2} \lambda^{3} \xi^{2} |\nabla^{2} \psi|^{2} + \lambda |\nabla ^{3} \psi|^{2} \bigr] d \Gamma dt 
    - C \mathbb{E} \int_{\Sigma}  s \lambda^{3} \xi  |\nabla^{2} \psi|^{2}  d \Gamma dt
    \\\notag
    & \geq
    - C \mathbb{E} \int_{\Sigma}  \bigl[  s^{2} \lambda^{3} \xi^{2} |\nabla^{2} \psi|^{2} + \lambda |\nabla ^{3} \psi|^{2} \bigr] d \Gamma dt 
    \\
    & \geq
    \mathcal{B}_{1}+\mathcal{B}_{2}.
\end{align}
Hence, from \cref{eqCarlemanBoundary10,eqCarlemanBoundary9}, we obtain 
\begin{align}
    \label{eqCarlemanBoundary11}
    \mathbb{E}\int_{Q} \operatorname{div}  (V_{1} + V_{2} ) dxdt 
    \geq
    - C  \mathbb{E} \int_{0}^{T} e^{2s\alpha_{\star}} | \varphi |_{H^{4}(G)}^{2} dt 
    + \mathcal{B}_{1}+\mathcal{B}_{2}
\end{align}

{\bf Step 2}. In this step, we deal with the interior term $ \mathbb{E}\int_{Q} M dxdt $ and stochastic terms via integration by parts.

From \cref{thm.fundamentalIndentity}, we have 
\begin{align} \label{eq.paM1}
    \notag
     M   & =
    8 s \lambda^{2}\xi   | \nabla \Delta \psi \cdot \nabla \eta  | ^{2}
    + 32 s^{3} \lambda^{4} \xi^{3}  | \nabla^{2} \psi  \nabla \eta \nabla \eta   |^{2}
    + 48 s^{3} \lambda^{3} \xi^{3} \nabla^{2} \eta  (\nabla^{2} \psi \nabla \eta )  ( \nabla^{2} \psi \nabla \eta  )
    \\  \notag
    & \quad
    + 16 s^{3} \lambda^{3} \xi^{3}  ( \nabla^{2} \psi \nabla \eta \nabla \eta  ) \sum\limits_{i,j=1}^{n} \eta_{x_{i}x_{j}} \psi_{x_{i}x_{j}}
    -16 s^{3} \lambda ^{4} \xi ^{3} |\nabla \eta|^{2}  | \nabla^{2} \psi \nabla \eta  | ^{2}
    +6 s^{3} \lambda^{4} \xi ^{3} |\nabla \eta |^{4} | \nabla^{2} \psi|^{2}
    \\  \notag
    & \quad 
    + 4 s^{3} \lambda^{3} \xi ^{3}  ( \nabla^{2} \eta \nabla \eta \nabla \eta  ) |\nabla ^{2} \psi|^{2}
    + 2  s^{3} \lambda^{3} \xi ^{3} |\nabla \eta|^{2} \Delta \eta |\nabla^{2} \psi|^{2}
    + 2 s^{3} \lambda^{4} \xi ^{3} |\nabla \eta |^{4} | \Delta \psi|^{2}
    \\[2mm] \notag
    & \quad
    - 4 s^{3} \lambda^{3} \xi ^{3}  ( \nabla^{2} \eta \nabla \eta \nabla \eta  ) |\Delta \psi|^{2}
    - 2  s^{3} \lambda^{3} \xi ^{3} |\nabla \eta|^{2} \Delta \eta |\Delta \psi|^{2}
    + 40 s^{5} \lambda^{6} \xi^{5} |\nabla \eta|^{4} | \nabla \psi \cdot \nabla \eta|^{2}
    \\[2mm] \notag
    & \quad
    + 64 s^{5} \lambda ^{5} \xi^{5}  ( \nabla^{2} \eta \nabla \eta \nabla \eta  ) |\nabla \psi \cdot \nabla \eta|^{2}
    - 16 s^{5} \lambda^{6} \xi^{5} |\nabla \eta|^{6} |\nabla  \psi|^{2}
    + 8 s^{7} \lambda ^{8} \xi^{7} |\nabla \eta|^{8} \psi^{2}
    + \lambda |\nabla \Delta \psi|^{2}
    \\[2mm] \notag
    & \quad
    - s^{3} \lambda^{\frac{7}{2}} \xi^{3} |\nabla \eta|^{4}|\Delta \psi|^{2}
    + 4 s^{5} \lambda^{\frac{11}{2}} \xi^{5} |\nabla \eta|^{4}|\nabla \psi \cdot \nabla \eta|^{2}
    + 2 s^{5} \lambda^{\frac{11}{2}} \xi^{5} |\nabla \eta|^{6}|\nabla \psi|^{2}
    -  s^{7} \lambda^{\frac{15}{2}} \xi^{7} |\nabla \eta|^{8}\psi^{2}
    \\[2mm]
    & \quad
    - 4 s^{2} \lambda^{3} \xi^{2} |\nabla ^{2} \psi \nabla \eta|^{2}
    - 2 s^{2} \lambda^{3} \xi^{2} |\nabla \eta|^{2}|\Delta \psi|^{2}
    +  s^{4} \lambda^{5} \xi^{4} |\nabla \eta|^{4}|\nabla \psi |^{2}
    .
\end{align}
Thanks to  \cref{eq.pfthm2.1.S8}, it holds that 
\begin{align}\label{eq.MK0} 
    & \mathbb{E} \int_{Q} \big (  
        - 4 s^{2} \lambda^{3} \xi^{2} |\nabla ^{2} \psi \nabla \eta|^{2}
        - 2 s^{2} \lambda^{3} \xi^{2} |\nabla \eta|^{2}|\Delta \psi|^{2}
        +  s^{4} \lambda^{5} \xi^{4} |\nabla \eta|^{4}|\nabla \psi |^{2}
    \big ) d x d t 
    \geq
    -  \mathcal{A}_{2} - \mathcal{A}_{3}.
\end{align} 
Combining \cref{eqEstimatesOrder,eq.pfthm2.1.S8}, using integration by parts, for $ s\geq s_{2} $, we obtain
\begin{align}
    \label{eq.MK1}
    \notag
    & - \mathbb{E}\int_{Q}  16 s^{3} \lambda ^{4} \xi ^{3} |\nabla \eta|^{2}  | \nabla^{2} \psi \nabla \eta  | ^{2} dxdt
    \\ 
    &  
    \geq 
    -4 \mathbb{E}\int_{Q}  s \lambda^{2} \xi |\nabla \Delta \psi \cdot \nabla \eta|^{2} dxdt
    -16 \mathbb{E}\int_{Q}  s^{5} \lambda^{6} \xi^{5} |\nabla \eta|^{4} |\nabla  \psi \cdot \nabla \eta|^{2} dxdt 
    -  \mathcal{A}_{2} 
    .
\end{align}
Thanks to \cref{eq.eta}, there exists  $ \lambda_{1} \geq C   $, such that for all  $\lambda \geq \lambda_{1} $, it holds that
\begin{align}\label{eq.MK2} 
    \notag
    &  \mathbb{E}  \int_{Q}  64 s^{5} \lambda ^{5} \xi^{5}  ( \nabla^{2} \eta \nabla \eta \nabla \eta  ) |\nabla \psi \cdot \nabla \eta|^{2} d x d t
    \\ 
    &  
    \geq  
    - C  \mathbb{E}  \int_{(0, T) \times G_{3}}s^{5} \lambda ^{5} \xi^{5} |\nabla \psi|^{2} d x d t 
    -  \mathbb{E}  \int_{Q} s^{5} \lambda^{6} \xi^{5} |\nabla \eta|^{4} | \nabla \psi \cdot \nabla \eta|^{2} d x d t 
    .
\end{align}
Combining \cref{eqEstimatesOrder,eq.pfthm2.1.S8}, using integration by parts, for $ s\geq s_{2} $, we have 
\begin{align}\label{eq.s2-3-2}  
     \mathbb{E}  \int_{Q} \big[4 s^{3} \lambda^{3} \xi ^{3}  ( 
        \nabla^{2} \eta \nabla \eta \nabla \eta  ) 
    +2  s^{3} \lambda^{3} \xi ^{3} |\nabla \eta|^{2} \Delta \eta\big] ( |\nabla ^{2} \psi|^{2}-|\Delta \psi|^{2} ) d x d t
     \geq - \mathcal{A}_{2} 
    ,
\end{align}
and 
\begin{align}
    \label{eq.s2-3-2-1}
      6 \mathbb{E}  \int_{Q} s^{3} \lambda^{4} \xi^{3} | \nabla \eta |^{4} |\nabla^{2} \psi|^{2} d x d t 
      \geq 
    6 \mathbb{E}  \int_{Q} s^{3} \lambda^{4} \xi^{3} | \nabla \eta |^{4} | \Delta \psi|^{2} d x d t
    - \mathcal{A}_{2} .
\end{align}
Similarly, for $ s \geq s_{2} $, we deduce 
\begin{align} \label{eq.MK4}
    \notag
    & \mathbb{E}  \int_{Q} 
    16 s^{3} \lambda^{3} \xi^{3}  ( \nabla^{2} \psi \nabla \eta \nabla \eta  )\sum\limits_{i,j=1}^{n} \eta_{x_{i}x_{j}} \psi_{x_{i}x_{j}}    dxdt
    \\ 
    & 
    \geq  \mathbb{E} \int_{Q} 16  s^{3} \lambda^{3} \xi^{3} \nabla^{2} \eta  (\nabla^{2} \psi \nabla \eta )  ( \nabla^{2} \psi \nabla \eta  ) dxdt 
    - ( \mathcal{A}_{2}+\mathcal{A}_{3} )
    .
\end{align}
Combining \cref{eqEstimatesOrder,eq.pfthm2.1.S8}, there exists   $ \lambda_{2} \geq  C  $, such that for all  $  \lambda \geq \lambda_{2}$ and  $  s \geq s_{2} $, it holds that 
\begin{align}
    \label{eq.MK4p}
    \notag
    &  \mathbb{E}  \int_{Q} 64   s^{3} \lambda^{3} \xi^{3} \nabla^{2} \eta   (\nabla^{2} \psi \nabla \eta )   ( \nabla^{2} \psi \nabla \eta  ) dxdt
    \\
    &\geq - \mathbb{E}\int_{Q}  s \lambda^{2} \xi |\nabla \Delta \psi \cdot\nabla \eta|^{2} dxdt
    - \mathbb{E}\int_{Q}  s^{5} \lambda^{6} \xi^{5} |\nabla \eta|^{4} |\nabla  \psi \cdot\nabla \eta|^{2} dxdt
    -  \mathcal{A}_{2}  
    .
\end{align}

Next, we estimate the stochastic terms. Thanks to the divergence  theorem, we obtain 
\begin{align}
    \label{eqCarlemanBoundary12}
    - 2 \mathbb{E}\int_{Q} \sum_{i, j=1}^n (
   \psi_{x_{i}x_{i}x_{j}} d \psi - \psi_{x_{i}x_{j}}d \psi_{x_{i}}
   +\Psi_{2} \psi_{x_{i}} \delta_{ij} d \psi
   +\Psi_{3}^{ij} \psi_{x_{i}} d \psi
   )_{x_{j}} d x
   = 0
\end{align}
From \cref{eq.carlemanForwardEquation}, for $ s \geq s_{2} $, it holds that 
\begin{align}
    \label{eqCarlemanBoundary13}\notag
    & 
   - \mathbb{E}\int_{Q} \theta^{2} \sum_{i, j=1}^n\big[  d \varphi_{x_{i}x_{j}} + 2 \ell_{x_{j}} d \varphi_{x_{i}} +   ( \ell_{x_{i}x_{j}} + \ell_{x_{i}} \ell_{x_{j}} ) d \varphi \big]^{2} d x
   -  \mathbb{E}\int_{Q}  \Psi_{6} \theta^{2} ( d \varphi )^{2} d x  
    \\ \notag
    & \quad  
    + \mathbb{E}\int_{Q} \theta^{2}  \Psi_{2} | d \nabla \varphi
    + \nabla \ell d \varphi|^{2} d x
    + 4 \mathbb{E}\int_{Q}  s^{2}\lambda^{2}\xi^{2} \theta^{2} | \nabla \eta d \nabla \varphi
    + \nabla \eta \nabla \ell d \varphi|^{2} d x
    \\ 
    & \quad  
    \geq - C  \Big (s^{4}  \mathbb{E} \int_{Q}\lambda^{4} e^{2 s\alpha}  \xi^{4}g^{2} d x d t + s^{2}  \mathbb{E} \int_{Q}\lambda^{2} e^{2 s\alpha}  \xi^{2}|\nabla g|^{2} d x d t 
    + \mathbb{E} \int_{Q} e^{2 s\alpha}  |\nabla^{2} g|^{2} d x d t \Big )
    .
\end{align}

Combining \cref{eq.thm31S1,eqCarlemanBoundary11,eq.paM1,eq.MK0,eq.MK1,eq.MK2,eq.s2-3-2,eq.s2-3-2-1,eq.MK4,eq.MK4p,eqCarlemanBoundary12,eqCarlemanBoundary13}, there exist $  \lambda_{3} \geq \max\{\lambda_{1}, \lambda_{2}\} $ and  $ s_{3} \geq \max \{s_{1},s_{2}\} $, such that for all  $ \lambda \geq \lambda_{3} $ and $ s \geq s_{3} $, we obtain 
\begin{align}
    \label{eq.thm31S6-Middle2}
    \notag
    & 
    2 \mathbb{E}\int_{Q} \theta K_{1} (  d \varphi + \Delta^{2} \varphi dt ) d x
    +C   \mathbb{E} \int_{Q} \theta  
    ( 
        f^{2} 
        + s^{4} \lambda^{4} \xi^{4} g ^{2}
        + s^{2} \lambda^{2} \xi^{2}|\nabla g|  ^{2}
        +   |\nabla^{2} g| ^{2}
    )d x d t
   \\\notag
   & 
   + C  \mathbb{E}  \int_{(0, T) \times G_{3}}s^{5} \lambda ^{5} \xi^{5} |\nabla \psi|^{2} d x d t 
   + C  \mathbb{E} \int_{0}^{T} e^{2s\alpha_{\star}} | \varphi |_{H^{4}(G)}^{2} dt  
   + \mathcal{A} + \mathcal{B}
   \\ \notag
    & \geq
   2 \mathbb{E}\int_{Q} K_{1}^{2} dxdt
   + 2 \mathbb{E}\int_{Q} K_{1} K_{3} d x
   \\  \notag
    & \quad 
    + ~ 2 \mathbb{E}\int_{Q} (  
            3 s \lambda^{2}\xi   | \nabla \Delta \psi \cdot \nabla \eta  | ^{2} 
            + 32 s^{3} \lambda^{4} \xi^{3}  | \nabla^{2} \psi  \nabla \eta \nabla \eta   |^{2} 
            + 8 s^{3} \lambda^{4} \xi ^{3} |\nabla \eta |^{4} | \Delta \psi|^{2} 
        \\  \notag
        &  \quad \quad \quad \quad \quad 
            + 22 s^{5} \lambda^{6} \xi^{5} |\nabla \eta|^{4} | \nabla \psi \cdot \nabla \eta|^{2} 
            - 16 s^{5} \lambda^{6} \xi^{5} |\nabla \eta|^{6} |\nabla  \psi|^{2}
            + 8 s^{7} \lambda ^{8} \xi^{7} |\nabla \eta|^{8} \psi^{2}
            + \lambda |\nabla \Delta \psi|^{2}
        \\   
        &\quad \quad \quad \quad \quad 
        + 2 s^{5} \lambda^{\frac{11}{2}} \xi^{5} |\nabla \eta|^{6}|\nabla \psi|^{2}
            -  s^{7} \lambda^{\frac{15}{2}} \xi^{7} |\nabla \eta|^{8}\psi^{2}
            - s^{3} \lambda^{\frac{7}{2}} \xi^{3} |\nabla \eta|^{4}|\Delta \psi|^{2}
     ) dxdt.    
\end{align}

{\bf Step 3}.
In this step, we get the estimate of $ \varphi $.

Using integration by parts, we have 
\begin{align}\label{eq.s2-5-2} \notag
    &\mathbb{E} \int_{Q} \bigl[
        \big( 8 s^{3} \lambda^{4} \xi^{3}  - s^{3} \lambda^{\frac{7}{2}} \xi^{3}\big) |\nabla \eta|^{4} |\Delta \psi|^{2}
        - 2 \big( 8 s^{3} \lambda^{4} \xi^{3}  - s^{3} \lambda^{\frac{7}{2}} \xi^{3}\big) |\nabla \eta|^{6} s^{2} \lambda^{2} \xi^{2} |\nabla \psi|^{2}
        \\\notag
        & \qquad ~
        + \big( 8 s^{3} \lambda^{4} \xi^{3}  - s^{3} \lambda^{\frac{7}{2}} \xi^{3}\big) |\nabla \eta|^{8} s^{4} \lambda^{4} \xi^{4} |\psi|^{2}
    \bigr] dxdt
    \\
    & \geq \mathbb{E} \int_{Q}
        \big( 8 s^{3} \lambda^{4} \xi^{3}  - s^{3} \lambda^{\frac{7}{2}} \xi^{3}\big) |\nabla \eta|^{4} \big(\Delta \psi +s^{2} \lambda^{2} \xi^{2} |\nabla \eta|^{2} \psi \big)^{2}
    dxdt
    - \mathcal{A}_{1}
    ,
\end{align}
and 
\begin{align}\notag
    \label{eq.thm31S5E2}
    &\mathbb{E}  \int_{Q}  s^{3} \lambda^{4} \xi ^{3} \theta^{2} |\nabla \eta |^{4} | \Delta \varphi|^{2}  dxdt
    \\\notag
    &= \mathbb{E}  \int_{Q}  s^{3} \lambda^{4} \xi^{3} |\nabla \eta|^{4} \big(
        \Delta \psi
        -2 s \lambda \xi \nabla \eta \cdot \nabla \psi
        + s^{2} \lambda^{2} \xi^{2} |\nabla \eta|^{2} \psi
        - s \lambda^{2} \xi |\nabla \eta|^{2} \psi
        - s \lambda \xi |\Delta \eta| \psi
    \big)^{2}dxdt
    \\  
    & \leq
    3 \mathbb{E}  \int_{Q}   s^{3} \lambda^{4} \xi^{3} |\nabla \eta|^{4} \big(
        \Delta \psi
        \!+\! s^{2} \lambda^{2} \xi^{2} |\nabla \eta|^{2} \psi
    \big)^{2}dxdt
    \!+\! 12 \mathbb{E}  \int_{Q} s^{5} \lambda^{6} \xi^{5} |\nabla \eta|^{4} |\nabla \psi\cdot \nabla \eta|^{2}
    + \mathcal{A}_{1}
    .
\end{align}
We also obtain 
\begin{align}
    \label{eq.thm31S5E2-M1}
    \notag
    & \mathbb{E} \int_{Q} \lambda |\nabla \Delta \psi |^{2} d x d t 
    + C  \mathbb{E} \int_{Q} (
        s^{2} \lambda^{3} \xi^{2} \theta^{2} |\nabla^{2} \varphi|^{2}
        + s^{4} \lambda^{5} \xi^{4} \theta^{2} |\nabla \varphi|^{2}
        + s^{6} \lambda^{7} \xi^{6} \theta^{2} | \varphi|^{2}
    ) d x d t 
    \\  
    & \geq
    \mathbb{E} \int_{Q} \lambda \theta^{2} |\nabla \Delta \varphi |^{2} d x d t 
    .
\end{align}
Using \cref{eq.pfthm2.1.S8}, there exists  $ s_{4} \geq C (T^{1/2} + T) $, such that for $ s \geq s_{4} $, it holds that 
\begin{align}
    \label{eq.thm31S5E3}
    2 \mathbb{E} \int_{Q} K_{1} K_{3}dx & \geq -  \mathbb{E} \int_{Q} K_{1}^{2} dx dt - \mathbb{E} \int_{Q} K_{3}^{2} dxdt 
    \geq  -  \mathbb{E} \int_{Q} K_{1}^{2} dx dt  - \mathcal{A}
    .
\end{align}
Thanks to\cref{eq.eta}, we obtain 
\begin{align}
    \label{eq.thm31S5E6}
    \mathbb{E} \int_{(0, T) \times G_{3}} s^{5} \lambda ^{5} \xi^{5} |\nabla \psi|^{2} dxdt 
    \leq C  \mathbb{E}  \int_{(0, T) \times G_{3}}   (s^{7} \lambda^{7} \xi^{7}\theta^{2}| \varphi|^{2}+s^{5} \lambda^{5} \xi^{5}\theta^{2}|\nabla \varphi|^{2}  )dx dt.
\end{align}
and
\begin{align}
    \label{eq.thm31S5E7}
    2 \mathbb{E}\int_{Q} \theta K_{1} (   d \varphi + \Delta^{2} \varphi dt ) dx
    \leq \mathbb{E} \int_{Q} K_{1}^{2} dxdt + \mathbb{E} \int_{Q} \theta^{2} f^{2} dx dt
    .
\end{align}
Using \cref{lemma.LaplaceCarleman}, there exist $ \lambda_{4} \geq \hat{\lambda} $ and $ s_{4} \geq \hat{s} $, such that  for all $ \lambda \geq \lambda_{4} $ and $ s \geq s_{4} $, we obtain 
\begin{align}
    \label{eq.thm31S5E8_1}
    \notag
    &\mathbb{E} \int_{Q} s^{6}  \lambda^{8} \xi^{6} \theta^{2} |\varphi|^{2} dx dt
    +\mathbb{E} \int_{Q} s^{4}  \lambda^{6} \xi^{4} \theta^{2} | \nabla \varphi|^{2} dx dt
    \\
    &\leq
    C \Big(
    \mathbb{E} \int_{(0,T)\times G_{3}} s^{6}  \lambda^{8} \xi^{6} \theta^{2} |\varphi|^{2} dx dt
    + \mathbb{E} \int_{Q} s^{3}  \lambda^{4} \xi^{3} \theta^{2} |\Delta \varphi|^{2} dx dt
    \Big)
    .
\end{align}
By means of \cref{lemma.auxiliaryEstimates,eqEstimatesOrder}, we can absorb $ \mathcal{B}  $ and $ |e^{s \alpha_{*}} \varphi|_{L_{\mathbb{F}}^{2}(0, T ; H^{4}(G))}^{2} $.

Next, we estimate $ \mathcal{A}_{3} $.  
Letting  $ \tilde{\varphi} = s \lambda^{2} \xi e^{s \alpha} \varphi $, we obtain 
\begin{equation}\label{eq.thm31S5E11}
    \mathbb{E} \int_{Q} | \Delta \tilde{\varphi} |^{2} dxdt 
    \leq C \mathbb{E} \int_{Q}  (s^{6} \lambda^{8}  \xi^{6} \theta^{2}| \varphi|^{2}+s^{4} \lambda^{6}  \xi^{4} \theta^{2}|\nabla \varphi|^{2} + s^{2} \lambda^{4}  \xi^{2} \theta^{2}| \Delta \varphi|^{2} )dxdt.
\end{equation}
Note that 
\begin{equation}\label{eq.thm31S5E12}
        |\tilde{\varphi}|^{2}_{L^{2}_{\mathbb{F}}( 0,T; H^{2}(G) )} \leq C  |\Delta \tilde{\varphi}|^{2}_{L^{2}_{\mathbb{F}}( 0,T; L^{2}(G) )},
\end{equation}
and 
\begin{align}
    \label{eq.pfMpara3}
    & \mathbb{E} \int_{Q}  s^{2} \lambda^{4}  \xi^{2} \theta^{2}| \nabla^{2} \varphi|^{2}  dxdt 
    \leq 
    C |\tilde{\varphi}|^{2}_{L^{2}_{\mathbb{F}}( 0,T; H^{2}(G) )} 
    + C \mathbb{E} \int_{Q}  (
        s^{6} \lambda^{8}  \xi^{6} \theta^{2}| \varphi|^{2}
    +   s^{4} \lambda^{6}  \xi^{4} \theta^{2}|\nabla \varphi|^{2}  )dxdt.
\end{align}
Combining  \cref{eq.thm31S6-Middle2,eq.s2-5-2,eq.thm31S5E2-M1,eq.thm31S5E3,eq.thm31S5E6,eq.thm31S5E7,eq.thm31S5E8_1,eq.thm31S5E11,eq.thm31S5E12,eq.pfMpara3}, there exist $ \lambda_{5} \geq \max \{ \lambda_{3}, \lambda_{4} \}$  and $ s_{5} \geq \max \{ s_{3}, s_{4} \} $, such that for all $ \lambda \geq \lambda_{5} $ and $ s \geq s_{5} $, it holds that 
\begin{align}
    \label{eq.thm31S5E13}
    \notag
    & C  \mathbb{E} \int_{Q} \theta^{2} f^{2} dx dt    
    + C  \mathbb{E} \int_{Q} \theta^{2}
    (
      s^{5} \lambda^{5} \xi^{5} g ^{2}
      + s^{3} \lambda^{3} \xi^{3}|\nabla g|  ^{2}
      + s \lambda \xi |\nabla^{2} g| ^{2}
    )d x d t
    \\\notag
    & 
    + C \Big[ \mathbb{E} \int_{(0, T) \times G_{3}} \big(s^{7} \lambda^{7} \xi^{7}\theta^{2}| \varphi|^{2}+s^{5} \lambda^{5} \xi^{5}\theta^{2}|\nabla \varphi|^{2} +s^{3} \lambda^{4} \xi^{3}\theta^{2}|\Delta \varphi|^{2}\big)d x d t\Big]   
    \\\notag
    & 
    + C  \mathbb{E} \int_{(0,T)\times G_{2}} s^{6}  \lambda^{8} \xi^{6} \theta^{2} |\varphi|^{2} dx dt  
    \\ 
    & \geq
    \mathbb{E} \int_{Q} \big (s^{6} \lambda^{8}  \xi^{6} \theta^{2}| \varphi|^{2}+s^{4} \lambda^{6}  \xi^{4} \theta^{2}|\nabla \varphi|^{2} +s^{3} \lambda^{4}  \xi^{3} \theta^{2}|\Delta \varphi|^{2}
    + s^{2} \lambda^{4}  \xi^{2} \theta^{2}| \nabla^{2} \varphi|^{2} 
    + \lambda \theta^{2} |\nabla \Delta \varphi |^{2}  
    \big ) d x d t.    
\end{align}

Now, we estimate the local terms on $ \nabla \varphi $ and $ \Delta \varphi $. Let us introduce a cut-off function $ \rho $ such that
\begin{equation}\label{eq.pfthm2.1.S31}
    \rho \in C_{0}^{\infty}(G_{2}), \quad \rho=1 \text { in } G_{3},  \quad 0 \leq \rho \leq 1.
\end{equation}
Using It\^o's formula and \cref{eq.carlemanForwardEquation}, we get
\begin{align*}
    d ( \rho^{4} \xi^{3} \theta^{2}\varphi^{2}  )
        = \big( 2s\rho^{4} \xi^{3} \alpha_{t}+3\rho^{4}\xi^{2}\xi_{t}  \big)\theta^{2}\varphi^{2} dt+  \rho^{4} \xi^{3}  \theta^{2}g^{2}dt
            +2\rho^{4} \xi^{3}  \theta^{2} \varphi d \varphi.
\end{align*}
Integrating it on $ Q_{2} $, taking mathematical expectation on the both sides, we arrive
\begin{align}
    \notag
    0=&\mathbb{E} \int_{Q_{2}}  \big( 2s\rho^{4} \xi^{3} \alpha_{t}+3\rho^{4}\xi^{2}\xi_{t}\big)\theta^{2}\varphi^{2} dxdt
    +\mathbb{E} \int_{Q_{2}} \rho^{4} \xi^{3} \theta^{2}g^{2}dxdt
    - \mathbb{E} \int_{Q_{2}}  2\rho^{4} \xi^{3} \theta^{2} \varphi\Delta^{2} \varphi dxdt
    \\ \label{eq.pfthm2.1.S33}
    &
    + \mathbb{E} \int_{Q_{2}}  2\rho^{4} \xi^{3} \theta^{2} \varphi f dxdt.
\end{align}
From \cref{eq.pfthm2.1.S31} and integration by parts, there exist $ \lambda_{6} \geq C $ and $ s_{6} \geq C T $, such that for all $ \lambda \geq \lambda_{6} $ and $ s \geq s_{6} $, it follows that
\begin{align}
    \label{eq.pfthm2.1.S34}
\notag &\mathbb{E} \int_{Q_{2}}
s^{3}\lambda^{4}\xi^{3} \rho^{4} \theta^{2}
|\Delta \varphi|^{2} dx dt\\
& \leq \mathbb{E} \int_{Q_{2}}  s^{3}
\lambda^{4} \xi^{3} \rho^{4}\theta^{2}
\varphi\Delta^{2} \varphi dxdt + C \mathbb{E} \int_{Q_{2}}
s^{5}\lambda^{6}\xi^{5} \rho^{4} \theta^{2}
|\nabla \varphi|^{2} dxdt
\\ 
\notag &\quad + C \mathbb{E} \int_{Q_{2}}
s^{7}\lambda^{8} \xi^{7} \theta^{2} \varphi^{2} dxdt
    + C \mathbb{E} \int_{Q}  s^{3}\lambda^{4}\xi^{3}  \theta^{2} ( |\varphi|^{2} + |\nabla \varphi|^{2} ) dxdt.
\end{align}
For $ s\geq s_{6} $, we also have
\begin{equation}\label{eq.pfthm2.1.S35}
        \Big| \mathbb{E} \int_{Q_{2}}  2 s^{3}\lambda^{4}\xi^{3} \rho^{4}  \theta^{2} \varphi f dxdt  \Big|
    \leq
    C \mathbb{E} \int_{Q_{2}} s^{6}\lambda^{8} \xi^{6} \theta^{2} \varphi^{2} dxdt
    +C\mathbb{E} \int_{Q}  \theta^{2} f^{2} dxdt
\end{equation}
and
\begin{equation}\label{eq.pfthm2.1.S36}
        \Big| \mathbb{E} \int_{Q_{2}}  s^{3}\lambda^{4}  ( 2s\rho^{4} \xi^{3} \alpha_{t}+3\rho^{4}\xi^{2}\xi_{t}  )\theta^{2}\varphi^{2} dxdt  \Big|
    \leq
    C   \mathbb{E} \int_{Q_{2}} s^{7}\lambda^{8} \xi^{7} \theta^{2} \varphi^{2} dxdt.
\end{equation}

Using integration by parts, for $ \varepsilon > 0 $,  $ \lambda \geq \lambda_{6} $ and $ s \geq s_{6} $, we obtain
\begin{align}
    \label{eq.pfthm2.1.S38}
    \mathbb{E} \int_{Q_{2}}  s^{5}\lambda^{6}\xi^{5} \rho^{4} \theta^{2} |\nabla \varphi|^{2} dxdt
    \leq \varepsilon \mathbb{E} \int_{Q_{2}}  s^{3}\lambda^{4}\xi^{3} \psi^{4} \theta^{2} |\Delta \varphi|^{2} dx dt +    C\varepsilon^{-1}\mathbb{E} \int_{Q_{2}} s^{7}\lambda^{8} \xi^{7} \theta^{2} \varphi^{2} dxdt.
\end{align}

From \cref{eq.pfthm2.1.S31,eq.pfthm2.1.S33,eq.pfthm2.1.S34,eq.pfthm2.1.S35,eq.pfthm2.1.S36,eq.pfthm2.1.S38}, letting $ \varepsilon>0 $ small enough,  for $ \lambda \geq  \lambda_{6} $ and $ s\geq s_{6} $, it follows that
\begin{align}
    \label{eq.pfthm2.1.S39}
    \notag
    &\mathbb{E}  \int_{(0,T)\times G_{3}}   (s^{5} \lambda^{5} \xi^{5}\theta^{2}|\nabla \varphi|^{2} +s^{3} \lambda^{4} \xi^{3}\theta^{2}|\Delta \varphi|^{2} )dxdt
    \\
    & \leq
    C \Big[\mathbb{E} \int_{Q}  \theta^{2} (f^{2} + s^{4} \lambda^{4} \xi^{4} g ^{2} ) dxdt
    + \mathbb{E} \int_{Q_{2}} s^{7}\lambda^{8} \xi^{7} \theta^{2} \varphi^{2} dxdt
    +  \mathbb{E} \int_{Q}  s^{3}\lambda^{4}\xi^{3}  \theta^{2} ( |\varphi|^{2} + |\nabla \varphi|^{2} ) dxdt
    \Big]
\end{align}

From \cref{eq.thm31S5E13}  and \cref{eq.pfthm2.1.S39}, for $ \lambda \geq  \max \{ \lambda_{5}, \lambda_{6} \} \geq C $ and $ s \geq \max \{ s_{5}, s_{6} \} \geq C (T + T^{1/2})$, we deduce \cref{eq.carlemanEstimate}.
\end{proof}

\section{Proof of the observability estimate} \label{sc.4}

In this section, we use the Carleman estimate \cref{thm.carlemanForwardPre} to prove \cref{thmObservabilityEstimateI}.

From \cref{thm.carlemanForwardPre}, we obtain the following proposition.
\begin{proposition} \label{prop1}
    Let $ (z_{1}, z_{2}) $ be a solution to the system \cref{eqEquationsAdj}. Then for all $ \lambda \geq \lambda_{0} $ and $ s \geq s_{0} $, it holds that 
    \begin{align*}
        \mathbb{E} \int_{Q} s^{6} \lambda^{8}  \xi^{6} e^{2 s \alpha} (z_{1}^{2} + z_{2}^{2}) d x d t
        \leq 
        C \mathbb{E} \int_{Q_{2}} s^{7} \lambda^{8} \xi^{7}  e^{2 s \alpha} (z_{1}^{2} + z_{2}^{2}) d x d t
        .
    \end{align*}
\end{proposition}

\begin{lemma} \label{lmS}
    For any constants $ c_{0} > 0 $  and  $ p, q, c_{1} \geq 0 $, there exists a constant    $ \tilde{s}_{0} > 0  $ such that for all $ s \geq \tilde{s}_{0} $, we have 
    \begin{align}\label{eqcoreLemma}
        |c_{1} s^{p} \xi^{q} e^{c_{0} s \alpha} | \leq 1
        .
    \end{align}
\end{lemma}
\begin{proof}
    Letting $ \gamma(t) = \frac{1}{t^{1/2}(T-t)^{1/2}} $, for $ t \in (0, T) $, it clearly holds that
    \begin{align} \label{eqGamma}
        \frac{2}{T} \leq \gamma(t)  < \infty .
    \end{align}
    By means of  \cref{eq.alpha,eq.eta}, the inequality
    \begin{align*}
        c_{1} s^{p} \xi^{q} e^{c_{0} s \alpha}  \leq c_{1} s^{p} c_{2}(\lambda) \gamma^{q} e^{- c_{0} c_{3}(\lambda) s \gamma} \leq 1
    \end{align*}
    holds if  
    \begin{align*}
        s \geq 
        \frac{\ln (c_{1} c_{2}(\lambda))}{ c_{0} c_{3}(\lambda) \gamma}
        + \frac{q \ln \gamma}{ c_{0} c_{3}(\lambda) \gamma}
        + \frac{p \ln s}{ c_{0} c_{3}(\lambda) \gamma}
        .
    \end{align*}
    Thanks to \cref{eqGamma} and $ \frac{\ln \gamma}{\gamma} \leq c_{4} < \infty $, it is sufficient that 
    \begin{align*}
        s \geq 
        \frac{ T \ln (c_{1} c_{2}(\lambda))}{ 2 c_{0} c_{3}(\lambda) }
        + \frac{q  c_{4} }{ c_{0} c_{3}(\lambda)  }
        + \frac{T p \ln s}{2  c_{0} c_{3}(\lambda)  }
        .
    \end{align*}
    Hence, there exists $ \tilde{s}_{0}=\tilde{s}_{0}(p, q, c_{1}, c_{0}, \lambda, T) $ such that \cref{eqcoreLemma} holds.
\end{proof}

\begin{proposition} \label{propTwotoOne}
    Fix $ \lambda = \lambda_{0} $. Let $ (z_{1}, z_{2}) $ be a solution to the system \cref{eqEquationsAdj}. There exists constants $ C, \tilde{s}_{1}  > 0 $ such that  for all $ s \geq \tilde{s}_{1} $, it holds that 
    \begin{align*}
        \mathbb{E} \int_{Q}   \xi^{6} e^{2 s \alpha} (z_{1}^{2} + z_{2}^{2}) d x d t
        \leq 
        C \mathbb{E} \int_{Q_{0}}   e^{ \frac{3}{2} s  \alpha}   z_{2}^{2}  d x d t
        .
    \end{align*}
\end{proposition}
\begin{proof}
    From \cref{as1}, without loss of generality, we assume that $ a_{2}(x, t) \geq \sigma $, a.e. $ (x, t) \in G_{1} \times (0,T) $, $ \mathbb{P}\rm{-a.s.} $ 

    Thanks to \cref{prop1,lmS}, there exists a constant $ \tilde{s}_{2} > 0$ such that for all $ s \geq \tilde{s}_{2} $, we have 
    \begin{align} \label{eqObEstimatePf0} \notag
        \mathbb{E} \int_{Q} s^{6}  \xi^{6} e^{2 s \alpha} (z_{1}^{2} + z_{2}^{2}) d x d t
        & \leq 
        C \mathbb{E} \int_{Q_{2}} s^{7}  \xi^{7}  e^{\frac{1}{3} s \alpha} e^{\frac{5}{3} s \alpha} (z_{1}^{2} + z_{2}^{2}) d x d t
        \\
        & \leq 
        C \mathbb{E} \int_{Q_{2}}   e^{\frac{5}{3} s \alpha} (z_{1}^{2} + z_{2}^{2}) d x d t
        .
    \end{align}

Let $ \zeta \in C_{0}^{\infty}(G_{0}) $ be a cut-off function satisfying $ \zeta = 1 \text{ in } G_{2} $ and $ 0 \leq \zeta \leq 1 $ in $ G_{1} $.
Put $ \rho = \zeta^{12} $. 
Let $ \beta, k, l, r $ and $ \tau $ be positive numbers which will be specified later.

Let 
\begin{align*}
    \mathcal{P} \deq 
        e^{k \tau \alpha} \rho^{\frac{4}{3}} z_{1}^{2}
        + \beta e^{2 \tau \alpha} \rho  z_{1} z_{2}
        + e^{l \tau \alpha} \rho^{\frac{2}{3}} z_{2}^{2}
    .
\end{align*}
Using It\^o's formula,  we have 
\begin{align*}
    d \mathcal{P} & = 
    k \tau  \alpha_{t} e^{k \tau \alpha} \rho^{\frac{4}{3}} z_{1}^{2}
    + 2 e^{k \tau \alpha} \rho^{\frac{4}{3}} z_{1} d z_{1}
    + e^{k \tau \alpha} \rho^{\frac{4}{3}}  (d z_{1})^{2}
    \\
    & \quad  
    + 2 \tau  \beta \alpha_{t} e^{2 \tau \alpha} \rho  z_{1} z_{2}
    + \beta e^{2 \tau \alpha} \rho  (z_{1} d z_{2} + z_{2} d z_{1} + d z_{1} d z_{2})
    \\
    & \quad  
    + l \tau  \alpha_{t} e^{l \tau \alpha} \rho^{\frac{2}{3}} z_{2}^{2}
    + 2  e^{l \tau \alpha} \rho^{\frac{2}{3}} z_{2} d z_{2}
    +   e^{l \tau \alpha} \rho^{\frac{2}{3}} (d z_{2})^{2}
    .
\end{align*}
From \cref{eqEquationsAdj}, noting that $ \lim\limits_{t \rightarrow 0+} \theta(x, t) = \lim\limits_{t \rightarrow T-} \theta(x, t) = 0 $ for all $ x \in G $, we have 
\begin{align} \label{eqObEstimatePf1} \notag
    & \beta  \mathbb{E} \int_Q e^{2 \tau \alpha} \rho a_2 z_{1}^2 d x d t
    \\ \notag
    & =   
       \mathbb{E} \int_Q \big(
           k \tau e^{k \tau \alpha} \alpha_t \rho^{\frac{4}{3}} z_{1}^2
           -2 e^{k \tau \alpha} \rho^{\frac{4}{3}} a_1 z_{1}^2
           +e^{k \tau \alpha} \rho^{\frac{4}{3}} a_3^2 z_{1}^2
       \big) d x d t 
       \\ \notag
       & \quad 
       -\mathbb{E} \int_Q \big[
           2 e^{k \tau \alpha} \rho^{\frac{4}{3}} b_1 z_{1} z_{2}
           -2 \beta \tau e^{2 \tau \alpha} \alpha_t \rho z_{1} z_{2}
           +\beta e^{2 \tau \alpha} \rho (a_1+b_2-a_3 b_3 ) z_{1} z_{2}
           +2  e^{l \tau \alpha} \rho^{\frac{2}{3}} a_2 z_{1} z_{2} 
       \big] d x d t 
       \\ \notag
       & \quad 
       + \mathbb{E} \int_Q \big( 
           l \tau e^{l \tau \alpha} \alpha_t \rho^{\frac{2}{3}} z_{2}^2
           -\beta e^{2 \tau \alpha} \rho b_1 z_{2}^2
           -2  e^{l \tau \alpha} \rho^{\frac{2}{3}} b_2 z_{2}^2
           + e^{l \tau \alpha} \rho^{\frac{2}{3}} b_3^2 z_{2}^2 
       \big) d x d t 
       \\ \notag
       & \quad 
       - \mathbb{E} \int_Q \big(
           2 e^{k \tau \alpha} \rho^{\frac{4}{3}} z_{1} \Delta^{2} z_{1}
           +\beta e^{2 \tau \alpha} \rho(z_{2} \Delta^{2} z_{1} +z_{1} \Delta^{2} z_{2})
           +2  e^{l \tau \alpha} \rho^{\frac{2}{3}} z_{2} \Delta^{2} z_{2}
       \big) d x d t
       \\
       & \deq 
       \sum_{i=1}^{4} I_{i}
       .
\end{align}

Next, we estimate $ I_{i} $ for $ i=1,2,3,4 $.

Thanks to \cref{eqEstimatesOrder,eqObEstimatePf1}, assuming $ k > 2 $ and  noting  \cref{lmS}, there exists $ \tau_{1} > 0$ such that for all $ \tau \geq \tau_{1} $, it holds that 
\begin{align} 
    \label{eqObEstimatePfI1}
    \notag
    I_{1} 
    & = \mathbb{E} \int_Q \big(
        k \tau e^{k \tau \alpha} \alpha_t \rho^{\frac{4}{3}} z_{1}^2
        -2 e^{k \tau \alpha} \rho^{\frac{4}{3}} a_1 z_{1}^2
        +e^{k \tau \alpha} \rho^{\frac{4}{3}} a_3^2 z_{1}^2
    \big) d x d t 
    \\ \notag
    & =
    \mathbb{E} \int_Q \big(
        k \tau e^{ (k-2) \tau \alpha} \alpha_t \rho^{\frac{1}{3}} 
        -2 e^{(k-2) \tau \alpha} \rho^{\frac{1}{3}} a_1 
        +e^{(k-2) \tau \alpha} \rho^{\frac{1}{3}} a_3^2 
    \big) e^{2 \tau \alpha} \rho z_{1}^{2} d x d t 
    \\
    & \leq 
    \mathbb{E} \int_Q  e^{2 \tau \alpha} \rho z_{1}^{2} d x d t 
    .
\end{align}

Using \cref{eqEstimatesOrder,eqObEstimatePf1} again, assuming $ k > 2 $, $ 2>r $ and $ l > 1+ \frac{r}{2} $, noting  \cref{lmS}  and Cauchy-Schwarz inequality, there exists $ \tau_{2} > 0$ such that for all $ \tau \geq \tau_{2} $, it holds that 
\begin{align}
    \label{eqObEstimatePfI2} \notag
    I_{2} 
    & = 
    -\mathbb{E} \int_Q \big[
        2 e^{k \tau \alpha} \rho^{\frac{4}{3}} b_1 z_{1} z_{2}
        -2 \beta \tau e^{2 \tau \alpha} \alpha_t \rho z_{1} z_{2}
        +\beta e^{2 \tau \alpha} \rho (a_1+b_2-a_3 b_3 ) z_{1} z_{2}
        +2  e^{l \tau \alpha} \rho^{\frac{2}{3}} a_2 z_{1} z_{2} 
    \big] d x d t 
    \\
    & \leq 
     \mathbb{E} \int_Q  e^{2 \tau \alpha} \rho z_{1}^{2} d x d t 
    + (1+ \beta^{2}) \mathbb{E} \int_Q e^{r \tau \alpha} \rho^{\frac{1}{3}}  z_{2}^2  d x d t 
    .
\end{align}

Utilizing \cref{eqEstimatesOrder,eqObEstimatePf1}, assuming $ l > r $ and $ r < 2 $,   noting  \cref{lmS}, there exists $ \tau_{3} > 0$ such that for all $ \tau \geq \tau_{3} $, it holds that 
\begin{align}
    \label{eqObEstimatePfI3} 
    \notag
    I_{3}
    & =
    \mathbb{E} \int_Q \big( 
        l \tau e^{l \tau \alpha} \alpha_t \rho^{\frac{2}{3}} z_{2}^2
        -\beta e^{2 \tau \alpha} \rho b_1 z_{2}^2
        -2  e^{l \tau \alpha} \rho^{\frac{2}{3}} b_2 z_{2}^2
        + e^{l \tau \alpha} \rho^{\frac{2}{3}} b_3^2 z_{2}^2 
    \big) d x d t 
    \\ \notag
    & =
    \mathbb{E} \int_Q \big( 
        l \tau e^{(l-r) \tau \alpha} \alpha_t \rho^{\frac{1}{3}}  
        -\beta e^{(2-r) \tau \alpha} \rho^{\frac{2}{3}} b_1  
        -2  e^{(l-r) \tau \alpha} \rho^{\frac{1}{3}} b_2  
        + e^{(l-r) \tau \alpha} \rho^{\frac{1}{3}} b_3^2
    \big) e^{r \tau \alpha} \rho^{\frac{1}{3}}  z_{2}^2  d x d t 
    \\
    & \leq 
    (1+ \beta) \mathbb{E} \int_Q e^{r \tau \alpha} \rho^{\frac{1}{3}}  z_{2}^2  d x d t 
    .
\end{align}

In order to estimate $ I_{4} $, through the direct  calculation, for the $  C^{2}(G) $ function $ \mathcal{R}, u $ and $  v  $,   we have the following equality 
\begin{align} \label{eqObEstimatePf2} \notag
    & \sum_{i,j=1}^{n}  \mathcal{R}  u v_{x_{i}x_{i}x_{j}x_{j}}  
\\
& = \sum_{i, j=1}^n \big( \mathcal{R} u  v_{x_{i}x_{i}x_{j}}  
- \mathcal{R}_{x_{j}}  u v_{x_{i}x_{i}}
-\mathcal{R} u_{x_{j}} v_{x_{i}x_{i}} 
 \big ) _{x_{j}}
+  \Delta \mathcal{R}  u \Delta v 
+ 2 \nabla \mathcal{R} \cdot \nabla u \Delta v 
+  \mathcal{R}  \Delta u \Delta v 
.
\end{align}
From \cref{eqObEstimatePf2}, choosing $ u = v = z_{1} $ and $ \mathcal{R} = - 2 e^{k \tau \alpha} \rho^{\frac{4}{3}} $, utilizing the divergence theorem, we have 
\begin{align} \label{eqObEstimatePfI4eq1} \notag
    - \mathbb{E} \int_{Q} 2 e^{k \tau \alpha} \rho^{\frac{4}{3}} z_{1} \Delta^{2} z_{1} dxdt
     & =
    - 2 \mathbb{E} \int_{Q}  \Delta \big(e^{k \tau \alpha} \rho^{\frac{4}{3}}\big) z_{1} \Delta z_{1} dxdt
    - 4 \mathbb{E} \int_{Q}   \nabla \big(e^{k \tau \alpha} \rho^{\frac{4}{3}}\big) \nabla z_{1} \Delta z_{1} dxdt
    \\
    & \quad 
    - 2 \mathbb{E} \int_{Q}   e^{k \tau \alpha} \rho^{\frac{4}{3}}   |\Delta z_{1} |^{2} dxdt
    .
\end{align}
Thanks to Cauchy-Schwarz inequality,  \cref{eq.alpha,lmS}, recalling that $ k > 2 $, there exists $ \tau_{4} > 0$ such that for all $ \tau \geq \tau_{4} $, it holds that 
\begin{align}
    \label{eqObEstimatePfI4eq2}
    - 2 \mathbb{E} \int_{Q}  \Delta \big(e^{k \tau \alpha} \rho^{\frac{4}{3}}\big) z_{1} \Delta z_{1} dxdt
    & \leq 
    \frac{1}{8}\mathbb{E} \int_{Q}  e^{k \tau \alpha} \rho^{\frac{4}{3}}   |\Delta z_{1}|^{2} dxdt
    + \mathbb{E} \int_{Q}  e^{2 \tau \alpha} \rho    | z_{1}|^{2} dxdt
    ,
\end{align}
and
\begin{align}
    \label{eqObEstimatePfI4eq3}
    - 4 \mathbb{E} \int_{Q}   \nabla \big(e^{k \tau \alpha} \rho^{\frac{4}{3}}\big) \nabla z_{1} \Delta z_{1} dxdt
    & \leq 
    C \mathbb{E} \int_{Q}  e^{k \tau \alpha} \rho^{\frac{7}{6}}    | \nabla z_{1}|^{2} dxdt
    + \frac{1}{8}\mathbb{E} \int_{Q}  e^{k \tau \alpha} \rho^{\frac{4}{3}}   |\Delta z_{1}|^{2} dxdt
    .
\end{align}
Combining \cref{eqObEstimatePfI4eq1,eqObEstimatePfI4eq2,eqObEstimatePfI4eq3}, for $ \tau \geq \tau_{4} $, we have 
\begin{align} \label{eqObEstimatePf3} \notag
    - \mathbb{E} \int_{Q} 2 e^{k \tau \alpha} \rho^{\frac{4}{3}} z_{1} \Delta^{2} z_{1} dxdt
    & \leq 
     \mathbb{E} \int_{Q}  e^{2 \tau \alpha} \rho    | z_{1}|^{2} dxdt
    + C \mathbb{E} \int_{Q}  e^{k \tau \alpha} \rho^{\frac{7}{6}}    | \nabla z_{1}|^{2} dxdt
    \\
    & \quad 
    - \frac{7}{4} \mathbb{E} \int_{Q}   e^{k \tau \alpha} \rho^{\frac{4}{3}}   |\Delta z_{1} |^{2} dxdt
    .
\end{align}

From \cref{eqObEstimatePf2}, choosing $ u = v = z_{2} $ and $ \mathcal{R} = - 2 e^{l \tau \alpha} \rho^{\frac{2}{3}} $, recalling $ l > r $, similarly to \cref{eqObEstimatePf3}, there exists $ \tau_{5} > 0$ such that for all $ \tau \geq \tau_{5} $, we have 
\begin{align} \label{eqObEstimatePf4} \notag
    - \mathbb{E} \int_{Q} 2 e^{l \tau \alpha} \rho^{\frac{2}{3}} z_{2} \Delta^{2} z_{2} dxdt
    & \leq 
     \mathbb{E} \int_{Q}  e^{r \tau \alpha} \rho^{\frac{1}{3}}    | z_{2}|^{2} dxdt
    + C \mathbb{E} \int_{Q}  e^{l \tau \alpha} \rho^{\frac{1}{2}}    | \nabla z_{2}|^{2} dxdt
    \\
    & \quad 
    - \frac{7}{4} \mathbb{E} \int_{Q}   e^{l \tau \alpha} \rho^{\frac{2}{3}}   |\Delta z_{2} |^{2} dxdt
    .
\end{align}
From \cref{eqObEstimatePf2}, choosing  $ \mathcal{R} = - \beta e^{2 \tau \alpha} \rho $,  we have 
\begin{align} \label{eqObEstimatePfI4eq4} \notag
    & - \mathbb{E} \int_{Q} \beta e^{2 \tau \alpha} \rho(z_{2} \Delta^{2} z_{1} +z_{1} \Delta^{2} z_{2}) dxdt
    \\ \notag
    & =
    - \beta \mathbb{E} \int_{Q}  \Delta \big(e^{2 \tau \alpha} \rho \big) ( z_{1} \Delta z_{2} +z_{2} \Delta z_{1} ) dxdt
    - 2 \beta \mathbb{E} \int_{Q}   \nabla \big(e^{2 \tau \alpha} \rho\big) (\nabla z_{1} \Delta z_{2} +  \nabla z_{2} \Delta z_{1} ) dxdt
    \\
    & \quad 
    - \beta \mathbb{E} \int_{Q}   e^{2 \tau \alpha} \rho    \Delta z_{1} \Delta z_{2}  dxdt
    .
\end{align}
Thanks to Cauchy-Schwarz inequality,  \cref{eq.alpha,lmS}, assuming that $ k + l < 4 $, there exists $ \tau_{6} > 0$ such that for all $ \tau \geq \tau_{6} $, it holds that 
\begin{align}
    \label{eqObEstimatePfI4eq5} \notag
    & -   \beta \mathbb{E} \int_{Q}  \Delta \big(e^{2 \tau \alpha} \rho \big) ( z_{1} \Delta z_{2} +z_{2} \Delta z_{1} ) dxdt
    \\  \notag
    & \leq 
    \frac{1}{4}\mathbb{E} \int_{Q}  e^{k \tau \alpha} \rho^{\frac{4}{3}}   |\Delta z_{1}|^{2} dxdt
    + \frac{1}{4}\mathbb{E} \int_{Q}  e^{l \tau \alpha} \rho^{\frac{2}{3}}   |\Delta z_{2}|^{2} dxdt
    + \frac{1}{4} \mathbb{E} \int_{Q}  e^{k \tau \alpha} \rho    | z_{1}|^{2} dxdt
    \\
    & \quad 
    + \frac{1}{4} \mathbb{E} \int_{Q}  e^{l \tau \alpha}  \rho^{\frac{1}{3}} | z_{2}|^{2} dxdt
    ,
\end{align} 
\begin{align}
    \label{eqObEstimatePfI4eq6} \notag
    & - 2  \beta  \mathbb{E} \int_{Q}   \nabla \big(e^{2 \tau \alpha} \rho\big) (\nabla z_{1} \Delta z_{2} +  \nabla z_{2} \Delta z_{1} )  dxdt
    \\  \notag
    & \leq 
    \frac{1}{4}\mathbb{E} \int_{Q}  e^{k \tau \alpha} \rho^{\frac{4}{3}}   |\Delta z_{1}|^{2} dxdt
    + \frac{1}{4}\mathbb{E} \int_{Q}  e^{l \tau \alpha} \rho^{\frac{2}{3}}   |\Delta z_{2}|^{2} dxdt
    + \frac{1}{4} \mathbb{E} \int_{Q}  e^{k \tau \alpha}  \rho^{\frac{7}{6}}    | \nabla z_{1}|^{2} dxdt
    \\
    & \quad 
    + \frac{1}{4} \mathbb{E} \int_{Q}  e^{l \tau \alpha}  \rho^{\frac{1}{2}} | \nabla z_{2}|^{2} dxdt
    ,
\end{align}
and
\begin{align}
    \label{eqObEstimatePfI4eq7}  
     -  \beta \mathbb{E} \int_{Q}   e^{2 \tau \alpha} \rho    \Delta z_{1} \Delta z_{2}  dxdt
    & \leq 
    \frac{1}{4}\mathbb{E} \int_{Q}  e^{k \tau \alpha} \rho^{\frac{4}{3}}   |\Delta z_{1}|^{2} dxdt
    + \frac{1}{4}\mathbb{E} \int_{Q}  e^{l \tau \alpha} \rho^{\frac{2}{3}}   |\Delta z_{2}|^{2} dxdt
    .
\end{align}
Combining \cref{eqObEstimatePfI4eq4,eqObEstimatePfI4eq5,eqObEstimatePfI4eq6,eqObEstimatePfI4eq7}, for $ \tau  \geq \tau_{6} $, it holds that 
\begin{align} \label{eqObEstimatePf5} \notag
    & - \mathbb{E} \int_{Q} \beta  e^{2 \tau \alpha} \rho(z_{2} \Delta^{2} z_{1} +z_{1} \Delta^{2} z_{2}) dxdt
    \\  \notag
    & \leq 
    \frac{1}{2}\mathbb{E} \int_{Q}  e^{k \tau \alpha} \rho^{\frac{4}{3}}   |\Delta z_{1}|^{2} dxdt
    + \frac{1}{2}\mathbb{E} \int_{Q}  e^{l \tau \alpha} \rho^{\frac{2}{3}}   |\Delta z_{2}|^{2} dxdt
    + \frac{1}{4} \mathbb{E} \int_{Q}  e^{k \tau \alpha}  \rho^{\frac{7}{6}}    | \nabla z_{1}|^{2} dxdt
    \\
    & \quad 
    + \frac{1}{4} \mathbb{E} \int_{Q}  e^{l \tau \alpha}  \rho^{\frac{1}{2}} | \nabla z_{2}|^{2} dxdt
    .
\end{align}

Hence, from \cref{eqObEstimatePf1,eqObEstimatePf3,eqObEstimatePf4,eqObEstimatePf5}, for $ \tau \geq \tau_{7} = \max \{ \tau_{4}, \tau_{5}, \tau_{6} \} $, we have 
\begin{align}
    \label{eqObEstimatePf6} \notag
    I_{4}
    & = 
    - \mathbb{E} \int_Q \big(
        2 e^{k \tau \alpha} \rho^{\frac{4}{3}} z_{1} \Delta^{2} z_{1}
        + \beta e^{2 \tau \alpha} \rho(z_{2} \Delta^{2} z_{1} +z_{1} \Delta^{2} z_{2})
        +2  e^{l \tau \alpha} \rho^{\frac{2}{3}} z_{2} \Delta^{2} z_{2}
    \big) d x d t
    \\  \notag
    & \leq 
    - \frac{5}{4}\mathbb{E} \int_{Q}  e^{k \tau \alpha} \rho^{\frac{4}{3}}   |\Delta z_{1}|^{2} dxdt
    - \frac{5}{4}\mathbb{E} \int_{Q}  e^{l \tau \alpha} \rho^{\frac{2}{3}}   |\Delta z_{2}|^{2} dxdt
    + C \mathbb{E} \int_{Q}  e^{k \tau \alpha}  \rho^{\frac{7}{6}}    | \nabla z_{1}|^{2} dxdt
    \\
    & \quad 
    + C \mathbb{E} \int_{Q}  e^{l \tau \alpha}  \rho^{\frac{1}{2}} | \nabla z_{2}|^{2} dxdt
    + \mathbb{E} \int_{Q}  e^{2 \tau \alpha} \rho    | z_{1}|^{2} dxdt
    + \mathbb{E} \int_{Q}  e^{r \tau \alpha} \rho^{\frac{1}{3}}    | z_{2}|^{2} dxdt
    .
\end{align}

Through the   direct  calculation, for  $  C^{2}(G) $ functions $ \mathcal{R}  $ and $  u  $,   we have the following equality 
\begin{align} \notag
    \mathcal{R}  | \nabla u |^{2}   
    =
    \sum_{i=1}^{n} \bigg(\mathcal{R} u_{x_{i}} u - \frac{1}{2} \mathcal{R}_{x_{i}} u^{2}\bigg)_{x_{i}}
    + \frac{1}{2} \Delta \mathcal{R} u^{2} 
    - \mathcal{R} \Delta u u
    .
\end{align}
Choosing $ \mathcal{R} =  e^{k \tau \alpha}  \rho^{\frac{7}{6}}  $ and $ u = z_{1} $, from \cref{eqEquationsAdj,lmS},  there exists $ \tau_{8} > 0$ such that for all $ \tau \geq \tau_{8} $, it holds that 
\begin{align}
    \label{eqObEstimatePf7} \notag
    \mathbb{E} \int_{Q}  e^{k \tau \alpha}  \rho^{\frac{7}{6}}    | \nabla z_{1}|^{2} dxdt
    & =  
    \frac{1}{2} \mathbb{E} \int_{Q} \Delta \Bigl( e^{k \tau \alpha}  \rho^{\frac{7}{6}}  \Bigr) |z_{1}|^{2}  dxdt
    - \mathbb{E} \int_{Q}    e^{k \tau \alpha}  \rho^{\frac{7}{6}}  \Delta z_{1} z_{1}  
    dxdt
    \\
    & \leq 
      \mathbb{E} \int_{Q}  e^{2 \tau \alpha} \rho    | z_{1}|^{2} dxdt 
    + \frac{1}{4 C} \mathbb{E} \int_{Q}  e^{k \tau \alpha} \rho^{\frac{4}{3}}   |\Delta z_{1}|^{2} dxdt
    .
\end{align}
Similarly, there exists $ \tau_{9} > 0$ such that for all $ \tau \geq \tau_{9} $, we have 
\begin{align}
    \label{eqObEstimatePf8} \notag
    \mathbb{E} \int_{Q}  e^{l \tau \alpha}  \rho^{\frac{1}{2}} | \nabla z_{2}|^{2} dxdt
    & =  
    \frac{1}{2} \mathbb{E} \int_{Q} \Delta \Bigl(  e^{l \tau \alpha}  \rho^{\frac{1}{2}} \Bigr) |z_{2}|^{2}  dxdt
    - \mathbb{E} \int_{Q}     e^{l \tau \alpha}  \rho^{\frac{1}{2}}  \Delta z_{2} z_{2}  
    dxdt
    \\
    & \leq 
      \mathbb{E} \int_{Q}  e^{r \tau \alpha} \rho^{\frac{1}{3}}    | z_{2}|^{2} dxdt 
    + \frac{1}{4 C} \mathbb{E} \int_{Q}  e^{l \tau \alpha} \rho^{\frac{2}{3}}   |\Delta z_{2}|^{2} dxdt
    .
\end{align}
Combining \cref{eqObEstimatePf6,eqObEstimatePf7,eqObEstimatePf8},  for $ \tau \geq \tau_{10} = \max \{ \tau_{7}, \tau_{8}, \tau_{9} \} $, we obtain 
\begin{align}
    \label{eqObEstimatePf9}
    \notag
    I_{4}
    & \leq 
    -  \mathbb{E} \int_{Q}  e^{k \tau \alpha} \rho^{\frac{4}{3}}   |\Delta z_{1}|^{2} dxdt
    -  \mathbb{E} \int_{Q}  e^{l \tau \alpha} \rho^{\frac{2}{3}}   |\Delta z_{2}|^{2} dxdt
    + C \mathbb{E} \int_{Q}  e^{2 \tau \alpha} \rho    | z_{1}|^{2} dxdt
    \\
    & \quad 
    + C \mathbb{E} \int_{Q}  e^{r \tau \alpha} \rho^{\frac{1}{3}}    | z_{2}|^{2} dxdt
    .
\end{align}

From \cref{eqObEstimatePf9,eqObEstimatePfI3,eqObEstimatePfI2,eqObEstimatePfI1,eqObEstimatePf1}, for $ \tau \geq \tau_{11} = \max \{ \tau_{1}, \tau_{2}, \tau_{3}, \tau_{10} \} $, it holds that 
\begin{align*}
    \beta  \mathbb{E} \int_Q e^{2 \tau \alpha} \rho a_2 z_{1}^2 d x d t
    \leq 
    C \mathbb{E} \int_{Q}  e^{2 \tau \alpha} \rho    | z_{1}|^{2} dxdt
    + C (1+ \beta^{2}) \mathbb{E} \int_{Q}  e^{r \tau \alpha} \rho^{\frac{1}{3}}    | z_{2}|^{2} dxdt
    .
\end{align*}
Choosing $ \beta = 2 C $, recalling that $ a_{2}(x, t) \geq \sigma $, a.e. $ (x, t) \in G_{1} \times (0,T) $, $ \mathbb{P}\rm{-a.s.} $,  for $ \tau \geq \tau_{11} $,  we have 
\begin{align*}
    \mathbb{E} \int_Q e^{2 \tau \alpha} \rho  z_{1}^2 d x d t
    \leq 
    C  \mathbb{E} \int_{Q}  e^{r \tau \alpha} \rho^{\frac{1}{3}}    | z_{2}|^{2} dxdt
    .
\end{align*}
Choosing $ \tau = \frac{5}{6} s $ and $ r = \frac{9}{5} $, recalling that $ \zeta = 1 \text{ in } G_{2} $ and $ \zeta \in C_{0}^{\infty}(G_{0}) $, for $ s \geq \frac{6}{5} \tau_{11} $, it holds that 
\begin{align*}
    \mathbb{E} \int_{Q_{2}} e^{ \frac{5}{3} s \alpha}   z_{1}^2 d x d t
    \leq 
    C  \mathbb{E} \int_{Q_{0}}  e^{ \frac{3}{2} s  \alpha}     | z_{2}|^{2} dxdt
    .
\end{align*}
Combining it with \cref{eqObEstimatePf0}, letting $ \tilde{s}_{1} = \max \{\tilde{s}_{2}, \frac{6}{5} \tau_{11} \} $, we complete the proof.
\end{proof}

Now, we are in a position to prove \cref{thmObservabilityEstimateI}.

\begin{proof}[Proof of \cref{thmObservabilityEstimateI}]

Thanks to \cref{propTwotoOne}, fixing $ \lambda = \lambda_{0} $ and $ s = \tilde{s}_{1} $, for the solution $ (z_{1}, z_{2}) $ of the system \cref{eqEquationsAdj}, it holds that 
\begin{align}
    \label{eqProofOEI1}
    \mathbb{E} \int_{Q}   \xi^{6} e^{2 s \alpha} (z_{1}^{2} + z_{2}^{2}) d x d t
    \leq 
    C \mathbb{E} \int_{Q_{0}}   e^{ \frac{3}{2} s  \alpha}   z_{2}^{2}  d x d t
    .
\end{align}
Recalling \cref{eq.alpha}, we have 
\begin{align}
    \label{eqProofOEI2}
    \mathbb{E} \int_{Q}   \xi^{6} e^{2 s \alpha} (z_{1}^{2} + z_{2}^{2}) d x d t
    \geq 
    \min_{x \in \overline{G}}  \big ( \xi^{6}(x,T/2) + e^{2 s \alpha(x,T/4)}  \big ) 
    \mathbb{E} \int_{\frac{T}{4}}^{\frac{3T}{4}}  \int_{G}  (z_{1}^{2} + z_{2}^{2}) d x d t
    ,
\end{align}
and
\begin{align}
    \label{eqProofOEI3}
    \mathbb{E} \int_{Q_{0}}   e^{ \frac{3}{2} s  \alpha}   z_{2}^{2}  d x d t
    \leq 
    \max_{(t, x) \in \overline{Q}} e^{\frac{3}{2} s \alpha} 
    \mathbb{E} \int_{Q_{0}}     z_{2}^{2}  d x d t
    .
\end{align}
Combining \cref{eqProofOEI3,eqProofOEI2,eqProofOEI1}, we get 
\begin{align} 
    \label{eqProofOEI4} \notag
    \mathbb{E} \int_{\frac{T}{4}}^{\frac{3T}{4}}  \int_{G}  (z_{1}^{2} + z_{2}^{2}) d x d t
    & \leq 
    C \frac{\max_{(t, x) \in \overline{Q}} e^{\frac{3}{2} s \alpha} }{\min_{x \in \overline{G}}  \big ( \xi^{6}(x,T/2) + e^{2 s \alpha(x,T/4)}  \big )}
    \mathbb{E} \int_{Q_{0}}     z_{2}^{2}  d x d t
    \\
    & \leq 
    C  \mathbb{E} \int_{Q_{0}}     z_{2}^{2}  d x d t
    .
\end{align}

Thanks to  It\^o's formula, for $ 0 \leq t_{1}\leq t_{2} \leq T $, we have
\begin{align*}
    \notag
    & 
    \mathbb{E} \int_{G} ( z_{1}^{2}(t_{2}) + z_{2}^{2}(t_{2}) )dx 
    -\mathbb{E} \int_{G} ( z_{1}^{2}(t_{1}) + z_{2}^{2}(t_{1}) ) dx 
    =
    \mathbb{E} \int_{t_{1}}^{t_{2}} \int_{G} d( z_{1}^2 + z_{2}^{2})dx
    \\ \notag
    &
    =  
    - \mathbb{E} \int_{t_{1}}^{t_{2}} \int_{G} [ 
        2 z_{1} ( \Delta^{2} z_{1} +   a_{1}z_{1} + b_{1} z_{2})   
        + 2 z_{2} ( \Delta^{2} z_{2} +   a_{2}z_{1} + b_{2} z_{2})   
        - a_{3}^{2} z_{1}^{2}
        - b_{3}^{2} z_{2}^{2} 
    ] dxdt
    \\
    & \leq
    - C   \mathbb{E} \int_{t_{1}}^{t_{2}} \int_{G} ( z_{1}^{2} + z_{2}^{2} ) dxdt
        .
\end{align*}
Hence, using Gronwall inequality, for all $ t \in (0,T) $, we have 
\begin{align*}
    \mathbb{E} \int_{G} (z_{1}^{2}(T) + z_{2}^{2}(T)) d x 
    \leq 
    C \mathbb{E} \int_{G} (z_{1}^{2}(t) + z_{2}^{2}(t)) d x 
    .
\end{align*}
Combining it with \cref{eqProofOEI4}, we complete the proof.
\end{proof}

\section{Proof of the null controllability} \label{sc.5}

\begin{proof}[Proof of \cref{thmControlResult}]

We define the linear subspace $ \mathcal{Y} $ of $ L_{\mathbb{F}}^{2}(0, T ; L^{2}(G_{0}))  $  to be  
\begin{align*}
    \mathcal{Y} =\big\{ \chi_{G_{0}} z_{2}  \mid  (z_{1}, z_{2}) ~  \text{solves the system \cref{eqEquationsAdj}  with} (z_{1}^{0}, z_{2}^{0}) \in  L^{2}(G)  \times L^{2}(G)  \big \}
    ,
\end{align*}
and define a linear functional $F$ on $ \mathcal{Y}  $ as follows
\begin{align*}
    F(\chi_{G_{0}} z_{2} )= \mathbb{E}\int_{G} ( y_{1}^{T} z_{1}(T) + y_{2}^{T} z_{2}(T)) d x.
\end{align*}
Thanks to \cref{thmObservabilityEstimateI}, it holds that 
\begin{align}\label{eq.pfNullConS2-1}\notag
    |F(\chi_{G_{0}} z_{2} )| & =\Bigl| \mathbb{E}\int_{G} ( y_{1}^{T} z_{1}(T) + y_{2}^{T} z_{2}(T)) d x \Bigr|
    \\\notag
    &
    \leq 
    \Bigl( \mathbb{E}\int_{G} ( |y_{1}^{T}|^{2} + |y_{2}^{T}|^{2} ) d x \Bigr)^{\frac{1}{2}}
    \Bigl( \mathbb{E}\int_{G} ( | z_{1}(T)|^{2} + | z_{2}(T)|^{2} ) d x \Bigr)^{\frac{1}{2}}
    \\
    & \leq C  \Bigl( \mathbb{E}\int_{G} ( |y_{1}^{T}|^{2} + |y_{2}^{T}|^{2} ) d x \Bigr)^{\frac{1}{2}} |z_{2}|_{L^{2}_{\mathbb{F}}(0,T; L^{2}(G_{0}))}
    .
\end{align}
Hence, $ F $ is a bounded linear functional on $ \mathcal{Y} $. 
By Hahn-Banach theorem, $ F $ can be extended to a bounded linear functional on $L_{\mathbb{F}}^{2}(0, T ; L^{2}(G_{0})) $. 
For simplicity, we also denote $ F $ for this extension.
From Riesz representation theorem, there exists $u \in L_{\mathbb{F}}^{2}(0, T ; L^{2}(G_{0})) $ such that 
\begin{align}\label{eq.pfNullConS3-1}
    \mathbb{E}\int_{G} ( y_{1}^{T} z_{1}(T) + y_{2}^{T} z_{2}(T)) d x  =F( \chi_{G_{0}} z_{2} ) 
    = \mathbb{E} \int_{0}^{T} \int_{G_{0}} z_{2} u d x d t
    .
\end{align}

We claim that $ u $ is the desired control. 
In fact, thanks to It\^o's formula and \cref{eqEquations}, we know that 
\begin{align}
    \label{eq.pfNullConS4-1}
    \notag
    & \mathbb{E} \int_{G} y_{1}^{T} z_{1}(T) d x-\int_{G}  y_{1} (0) z_{1}^{0}  d x   
    = \mathbb{E} \int_{0}^{T} \int_{G} d(  y_{1} z_{1}) dx
    \\ \notag
    & 
    = \mathbb{E} \int_{0}^{T} \int_{G} (z_{1} d  y_{1} +  y_{1} d z_{1} + d  y_{1} d z_{1}) dx 
    \\ \notag
    & 
    = \mathbb{E} \int_{0}^{T} \int_{G} 
    [
        ( \Delta z_{1} \Delta  y_{1} + a_{1} z_{1}  y_{1} + a_{2} z_{1} y_{2} +   a_{3} Y_{1} z_{1}  )
        + ( -  \Delta z_{1} \Delta  y_{1} - a_{1} z_{1}  y_{1} - b_{1} z_{2} y_{1} )
        -  a_{3} z_{1} Y_{1}
    ]
    dx dt 
    \\ 
    & 
    =\mathbb{E} \int_{0}^{T} \int_{G}   (a_{2} z_{1} y_{2}  - b_{1} z_{2} y_{1} ) d x d t,
\end{align}
and
\begin{align}
    \label{eq.pfNullConS4-2}
    \notag
    & \mathbb{E} \int_{G} y_{2}^{T} z_{2}(T) d x-\int_{G}  y_{2} (0) z_{2}^{0}  d x   
    = \mathbb{E} \int_{0}^{T} \int_{G} d(  y_{2} z_{2}) dx
    \\ \notag
    & 
    = \mathbb{E} \int_{0}^{T} \int_{G} (z_{2} d  y_{2} +  y_{2} d z_{2} + d  y_{2} d z_{2}) dx 
    \\ \notag
    & 
    = \mathbb{E} \int_{0}^{T} \int_{G} 
    [
        ( \Delta z_{2} \Delta  y_{2} + b_{1} z_{2}  y_{2} + b_{2} z_{2} y_{2} +   b_{3} Y_{2} z_{2} + \chi_{G_{0}} u z_{2}  )
        + ( -  \Delta z_{2} \Delta  y_{2} - a_{2} z_{1}  y_{2} - b_{2} z_{2} y_{2} )
    \\ \notag
    & \qquad \qquad \quad  ~
        -  b_{3} z_{2} Y_{2}
    ]
    dx dt 
    \\ 
    & 
    =
    \mathbb{E} \int_{0}^{T} \int_{G}   (b_{1} z_{2} y_{2}   - a_{2} z_{2} y_{2} ) d x d t
    + \mathbb{E} \int_{0}^{T} \int_{G_{0}} u z_{2}  d x d t
    .
\end{align}
Combining \cref{eq.pfNullConS3-1,eq.pfNullConS4-1,eq.pfNullConS4-2}, for any $ (z_{1}^{0}, z_{2}^{0}) \in  L^{2}(G)  \times L^{2}(G)  $, we have 
\begin{align*}
    \int_{G}  y_{1} (0) z_{1}^{0}  d x + \int_{G}  y_{2} (0) z_{2}^{0}  d x   = 0.
\end{align*}
Hence,  this implies $ (y_{1}(0), y_{2}(0)) = 0 $ in $ G $, $ \mathbb{P}\rm{-a.s.} $, which completes the proof.
\end{proof}

\appendix

\section{Proof of Lemma 2.4}\label{secProofLemma}

We need the following regularity estimate.

\begin{lemma} \label{lemmaRegularityEstimate}
    There exists a constant $ C > 0 $, such that for all $ h \in L^{2}_{\mathbb{F}}(0, T; H^{4}(G)) $, $ f \in L^{2}_{\mathbb{F}}(0, T; L^{2}(G)) $ and $ g \in L^{2}_{\mathbb{F}}(0, T; H_{0}^{2}(G)) $ satisfying
    \begin{align*}
        \left \{
        \begin{alignedat}{2}
            & d h + \Delta^2 h d t =    f d t  + g d W(t) && \quad \text { in } Q, \\
            & h= \frac{\partial h}{\partial \nu} =0 && \quad \text { on } \Sigma,\\
            & h(0, \cdot)=0 && \quad \text { in } G,
        \end{alignedat}
        \right .
    \end{align*}
    we have 
    \begin{align*}
        |h|^{2}_{L^{2}_{\mathbb{F}}(0, T; H^{4}(G))} 
        \leq C ( 
            |f|^{2}_{L^{2}_{\mathbb{F}}(0, T; L^{2}(G))} 
            + |g|^{2}_{L^{2}_{\mathbb{F}}(0, T; H^{2}(G))} 
        )
        .
    \end{align*}
\end{lemma}

\begin{proof}

Let $ \mathsf{A} $ be an unbounded operator on $ L^{2}(G) $ as follows
\begin{equation*}
    \left\{\begin{aligned}
        &\mathcal{D}(\mathsf{A})= \bigg\{ \varphi \in H^{4}(G)  \mid  \varphi = \frac{\partial \varphi}{\partial \nu} = 0 \text{ on } \Gamma \bigg\},   \\
        &\mathsf{A} \varphi= \Delta^{2} \varphi,    \quad \forall \varphi \in \mathcal{D}(\mathsf{A}).
    \end{aligned}\right.
\end{equation*}
Denote by $0 <\mu_1<\mu_2\leq \cdots $ the eigenvalues of $\mathsf{A}$ and
by $\{\phi_{i}\}_{i=1}^{\infty} \subset \mathcal{D}(\mathsf{A}) $ the corresponding eigenfunctions, which serves as a  standard orthonormal basis of $ L^{2}(G) $.
By means of the regularity for solutions to the elliptic equations (see \cite[Chapter 2, Theorem 5.1]{Lions1972a}), we know that  $ \phi_{i} \in H^{8}(G) $, and
\begin{align*}
    \phi_{i} = |\nabla \phi_{i}|=|\Delta^{2} \phi_{i}| = |\nabla \Delta^{2} \phi_{i}| = 0
\end{align*}
on $ \Gamma $ for $ i=1,2,3, \cdots $

Consider the following stochastic differential equations 
\begin{equation}\label{eq.lemma41E1}
    \begin{cases}
        d c_{i}^{m} = \mu_{i} c_{i}^{m} dt + ( f, \phi_{i} )_{L^{2}(G)}  dt +  (g, \phi_{i})_{L^{2}(G)} dW(t),\\
        c_{i}^{m}(0)=0,
    \end{cases}
\end{equation}
where $ i =1,  \cdots, m $.
Clearly, we have $  ( f, \phi_{i} )_{L^{2}(G)}, ~ ( g, \phi_{i} )_{L^{2}(G)} \in L^{2}_{\mathbb{F}}(0,T;\mathbb{R}) $. 
Thanks to the classical theory of stochastic differential equations (see \cite[Theorem 3.2]{Lue2021a}), equation \cref{eq.lemma41E1}   exists a unique solution  $  c_{i}^{m} (\cdot)  \in L_{\mathbb{F}}^{2}(\Omega ; C([0, T] ; \mathbb{R})) $ for  $i=1, \ldots, m$.

Let $ h^{m} = \sum\limits_{i=1}^{m} c_{i}^{m}(t) \phi_{i} $. 
From \cref{eq.lemma41E1}, we have 
\begin{align}\label{eq.lemma41E3}
    d h^{m} = -\Delta^{2} h^{m} dt + f^{m} dt + g^{m} dW(t),
\end{align}
where $  f^{m} = \sum\limits_{i=1}^{m} ( f, \phi_{i} )_{L^{2}(G)} \phi_{i}  $ and $  g^{m} = \sum\limits_{i=1}^{m} ( g, \phi_{i} )_{L^{2}(G)} \phi_{i}  $. Note that $ h^{m} (0) \equiv 0 $.
Thanks to It\^o's formula, it holds that 
\begin{equation}\label{eq.lemma41E4}
    d( h^{m} \Delta^{2} h^{m} ) = d h^{m} \Delta^{2} h^{m} + h^{m} \Delta^{2} (d h^{m} ) + d h^{m} d (\Delta^{2} h^{m}). 
\end{equation}
Integrating it on $ Q $, taking expectation on both sides, we have 
\begin{align}\label{eq.lemma41E4_0}
    \mathbb{E} \int _{G} h^{m}(T) \Delta^{2} h^{m}(T) dx= \mathbb{E} \int_{Q} [ 
    d h^{m} \Delta^{2} h^{m} + h^{m} \Delta^{2} (d h^{m} ) + d h^{m} d (\Delta^{2} h^{m})  
    ] dx.
\end{align}
From the definition of $ h^{m} $ and $ \phi_{i} $, we obtain 
\begin{equation}\label{eq.lemma41E4_1}
    \mathbb{E} \int _{G} h^{m}(T) \Delta^{2} h^{m}(T) dx =   \sum_{i=1}^{m}  \mu_{i} \mathbb{E} ( c_{i}^{m}(T) )^{2} \geq 0.
\end{equation}
From \cref{eq.lemma41E3}, we get 
\begin{align}\label{eq.lemma41E5_1}
    \mathbb{E} \int_{Q} d h^{m} \Delta^{2} h^{m} dx 
     & \leq 
    -\frac{1}{2}  \mathbb{E} \int_{Q} |\Delta^{2} h^{m} |^{2} dxdt
    + 2 \mathbb{E} \int_{Q} |f|^{2} dxdt
    ,
\\ \notag
\label{eq.lemma41E5_2}
\mathbb{E} \int_{Q}  h^{m} \Delta^{2} (d h^{m} ) dx 
& =
\mathbb{E} \int_{Q} h^{m} ( - \Delta^{4} h^{m} ) dxdt
+ \mathbb{E} \int_{Q} h^{m} ( \Delta^{2} f^{m} ) dxdt
\\ 
& \leq
-\frac{1}{2}  \mathbb{E} \int_{Q} |\Delta^{2} h^{m} |^{2} dxdt
+ 2\mathbb{E} \int_{Q} |f|^{2} dxdt,
\end{align}
and
\begin{align}\label{eq.lemma41E5_3}
    \mathbb{E} \int_{Q}  d h^{m} d (\Delta^{2} h^{m}) dx 
        =   \mathbb{E} \int_{Q} g^{m} \Delta^{2} g^{m}   dxdt
        =   \mathbb{E} \int_{Q}  |\Delta  g^{m}|^{2}   dxdt
    .
\end{align}
Combining \cref{eq.lemma41E4_0,eq.lemma41E4_1,eq.lemma41E5_1,eq.lemma41E5_2,eq.lemma41E5_3}, we have 
\begin{equation}\label{eq.lemma41E6}
    |\Delta^{2} h^{m}|^{2}_{L^{2}_{\mathbb{F}}( 0,T; L^{2}(G) )} 
    \leq 
    4 |f^{m}|^{2} _{L^{2}_{\mathbb{F}}( 0,T; L^{2}(G) )}
    + |\Delta g^{m}|^{2} _{L^{2}_{\mathbb{F}}( 0,T; L^{2}(G) )}.
\end{equation}
From \cref{eq.lemma41E6}, thanks to elliptic regularity theory (see \cite[Chapter 4, Theorem 4.1]{Lions1972b}), we get
\begin{align*}
    |h^{m}|^{2}_{L^{2}_{\mathbb{F}}( 0,T; H^{4}(G) )} \leq C
    |f^{m}|^{2} _{L^{2}_{\mathbb{F}}( 0,T; L^{2}(G) )}
    + C |\Delta g^{m}|^{2} _{L^{2}_{\mathbb{F}}( 0,T; L^{2}(G) )}.
\end{align*}
Therefore $ \{h^{m}\}_{m=1}^{\infty} $  exists a subsequence, which is still denoted by $ \{h^{m}\}_{m=1}^{\infty} $,
weakly convergence to $ \tilde{h} \in L^{2}_{\mathbb{F} } ( 0,T; H^{4}(G) ) $ as $m$ tends to $\infty$. 
On the other hand, thanks to the well-posedness of the  stochastic evolution equations (see
\cite[Theorem 3.13]{Lue2021a}), we know that $\lim\limits_{m\to\infty} |h^{m}- h|_{L^{2}_{\mathbb{F}}( 0,T; L^{2}(G)) } = 0$. 
Hence, $ \tilde{h}=h $ for a.e. $(t,x,\omega)\in [0,T]\times G\times\Omega$ and
\begin{equation*}
    |h|_{L^{2}_{\mathbb{F} } ( 0,T; H^{4}(G) )} \leq \liminf_{m \rightarrow \infty} |h^{m}|_{L^{2}_{\mathbb{F} } ( 0,T; H^{4}(G) )} 
    \leq C (|f|_{L^{2}_{\mathbb{F} } ( 0,T; L^{2}(G) )} + |g|_{L^{2}_{\mathbb{F} } ( 0,T; H^{2}(G) )}).
\end{equation*}
\end{proof}

\begin{proof}[Proof of \cref{lemma.auxiliaryEstimates}]
Thanks to the trace theorem (see \cite[Theorem 9.4]{Lions1972a}) and interpolation inequalities for Sobolev spaces, (see \cite[Theorem 9.1]{Lions1972a}), it holds that 
\begin{equation}\label{eq.lemma2E1}
\begin{gathered}
|\nabla^{2}\varphi |^{2}_{L^{2}(\Gamma)}   \leq
C|\varphi |^{2}_{H^{\frac{5}{2}}(G)}\leq  C
|\varphi|_{H^{2}(G)} |\varphi|_{H^{3}(G)}, \\
    |\nabla^{3}\varphi |^{2}_{L^{2}(\Gamma)}  \leq C|\varphi
|^{2}_{H^{\frac{7}{2}}(G)}\leq C |\varphi|_{H^{3}(G)}
|\varphi|_{H^{4}(G)}.
\end{gathered}
\end{equation}
Hence, combining it with \cref{eq.alphastar}, we have 
\begin{align}\label{eq.lemma2E2}
    \notag
	&\mathbb{E} \int_{\Sigma}  s^{\frac{9}{4}} \lambda^{3} \xi^{\frac{9}{4}}  \theta^{2} |\nabla^{2} \varphi|^{2} d\Gamma dt 
    + \mathbb{E} \int_{\Sigma}  s^{\frac{3}{4}} \lambda \xi^{\frac{3}{4}}  \theta^{2} |\nabla^{3} \varphi|^{2} d\Gamma dt
    \\
    & 
    \leq C \mathbb{E} \int_{Q} s^{6} \lambda^{8} \xi^{6} \theta^{2} \varphi^{2} dxdt
    + C \mathbb{E} \int_{0}^{T} e^{2 s \alpha_{\star}} |\varphi|^{2} _{H^{4}(G)} dt.
\end{align}
Let $ \varphi_{\star}\deq e^{s \alpha_{\star}} \varphi$. 
From \cref{eqauxiliaryEstimatesEq1}, we get 
\begin{equation*}
	\left\{ 
		\begin{alignedat}{2}
			& d \varphi_{\star} +  \Delta^2 \varphi_{\star} d t =  ( e^{s\alpha_{\star}} f + ( e^{s\alpha_{\star}} )_{t} \varphi ) d t+ e^{s\alpha_{\star}} g d W(t) && \quad \text { in } Q, \\
			& \varphi_{\star}= \frac{\partial \varphi_{\star}}{\partial \nu}=0 && \quad \text { on } \Sigma,\\
			& \varphi_{\star}(T, \cdot)=0 && \quad \text { in } G.
		\end{alignedat}
	\right.
\end{equation*}
By means of \cref{lemmaRegularityEstimate} and \cref{eqEstimatesOrder}, for $ s \geq C T^{\frac{1}{2}} $, we have 
\begin{align}
    \label{eq.lemma2E5}
	\notag
	& \mathbb{E} \int_{0}^{T} e^{2 s \alpha_{\star}} |\varphi|^{2} _{H^{4}(G)} dt 
    = |\varphi_{\star}|^{2}_{ L_{\mathbb{F}}^{2}(0, T ; H^{4}(G))} \\\notag
	& \leq C |e^{s\alpha_{\star}} f + ( e^{s\alpha_{\star}} )_{t} \varphi|^{2}_{L_{\mathbb{F}}^{2}(0, T ; L^{2}(G)) } 
    + C |e^{s\alpha_{\star}} g|_{L^{2}_{\mathbb{F}}(0, T; H^{2}(G))}^{2}
    \\
	& \leq C  \mathbb{E} \int_{Q}   \theta^{2} (
        |f|^{2}
        + | \Delta g| ^{2}
        + s^{2} \lambda^{2} \xi^{2}|\nabla g|  ^{2}
        +  s^{4} \lambda^{4} \xi^{4} g ^{2}
     ) dxdt + C  \mathbb{E}\int_{Q}  s^{6} \xi^{6} \theta^{2} \varphi^{2} dxdt
    .
\end{align}
Combining \cref{eq.lemma2E2,eq.lemma2E5}, we complete the proof.
\end{proof}

\end{document}